\documentclass[final]{siamltex}
\usepackage{amssymb,amsbsy,amsmath,amsfonts,amssymb,graphicx,color}
\usepackage{geometry} 

\date{\today}
\title{Generic Controllability of 3D Swimmers in a Perfect Fluid}
\author{Thomas
Chambrion\thanks{E-mail: \texttt{thomas.chambrion@iecn.u-nancy.fr}  and  \texttt{alexandre.munnier@iecn.u-nancy.fr}} \and Alexandre
Munnier$^{\ast}$\thanks{Both authors are with Institut \'Elie Cartan UMR 7502, Nancy-Universit\'e,
CNRS, INRIA, B.P.~239, F-54506 Vandoeuvre-l\`es-Nancy Cedex,
France, and INRIA
Nancy Grand Est, Projet CORIDA. Authors both supported by CPER MISN AOC. First author supproted by ANR GCM, ERC Boscain and INRIA color CUPIDSE, and second author by ANR CISIFS, GAOS and MOSICOB.}}
\begin{document}
\maketitle
\begin{abstract}
We address the problem of controlling a dynamical system governing the motion of a 3D  weighted shape changing body swimming in a perfect fluid. The rigid displacement of the swimmer results from the exchange of momentum between prescribed shape changes and the flow, the total impulse of the fluid-swimmer system being constant for all times.
We prove the following tracking results:
(i) Synchronized swimming: Maybe up to an arbitrarily small change of its density, any swimmer can approximately follow any given trajectory while, in addition, undergoing approximately any given shape changes. In this statement, the control consists in arbitrarily small superimposed deformations; 
(ii) Freestyle swimming:
Maybe up to an arbitrarily small change of its density, any swimmer can approximately tracks any given trajectory by combining suitably at most five basic movements that can be generically chosen (no macro shape changes are prescribed in this statement). 
\end{abstract}



\section{Introduction}
\subsection{Context}
Researches on bio-inspired locomotion in fluid have now a long history. Focusing on the area of Mathematical Physics, the modeling leads to a system of PDEs (governing the fluid flow) coupled with a system of ODEs (driving the rigid motion of the immersed body). The first difficulty mathematicians came up against was to prove the well-posedness of such systems. This task was carried out in \cite{San-Martin:2008ab} (where the fluid is assumed to be viscous and governed by Navier Stokes equations), in \cite{Munnier:2008aa} (for an inviscid fluid with potential flow) and in \cite{Dal-Maso:2010aa} (for low Reynolds numbers swimmers, the flow being governed by the stationary Stokes equations).  

After the well-posedness of the fluid-swimmer dynamics were established, the following step was to investigate its controllability. On this topic, still  few theoretical results are available: In \cite{Alouges:2008aa}, the authors prove that a 3D three-sphere mechanism, swimming along a straight line in a viscous fluid is controllable. In \cite{Chambrion:2010aa}, we prove that a generic 2D example of shape changing body swimming in a potential flow can track approximately any given trajectory.

Some authors are rather interested in describing the dynamics of swimming in terms of Geometric Mechanics (within the general framework presented for instance in \cite{Marsden:1999aa}). We refer to \cite{Kanso:2005aa} and the very recent paper \cite{Kelly:2011aa} for references in this area. 

In this article, we consider a 3D shape changing body swimming in a potential flow. Under some symmetry assumptions (the swimmer is alone in the fluid and the fluid-swimmer system fills the whole space) we prove generic controllability results, generalizing and improving what has been obtained for a particular 2D model in \cite{Chambrion:2010aa}.
\subsection{Modeling}
\subsubsection*{Kinematics}
We assume that the swimmer is the only immersed body in the fluid and that the fluid-swimmer system fills the whole space, identified with $\mathbf R^3$. Two frames are required in the modeling, the first one $\mathfrak E:=(\mathbf E_1,\mathbf E_2,\mathbf E_3)$ is fixed and Galilean and the second one $\mathfrak e:=(\mathbf e_1,\mathbf e_2,\mathbf e_3)$ is moving with its origin lying at any time at the center of mass of the swimming body. At any moment, there exist a rotation matrix $R\in {\rm SO}(3)$ and a vector $\mathbf r\in\mathbf R^3$ such that, if $X:=(X_1,X_2,X_3)^t$ and $x:=(x_1,x_2,x_3)^t$ are the coordinates of a same vector in respectively $\mathfrak E$ and $\mathfrak e$, then the equality $X=Rx+\mathbf r$ holds. The matrix $R$ is supposed to give also the {\it orientation} of the swimmer. The rigid displacement of the swimmer, on a time interval $[0,T]$ ($T>0$), is thoroughly described by the functions $t:[0,T]\mapsto R(t)\in{\rm SO}(3)$ and $t:[0,T]\mapsto \mathbf r(t)\in\mathbf R^3$, which are the unknowns of our problem. Denoting their time derivatives by $\dot R$ and $\dot{\mathbf r}$, we can define  the linear velocity $\mathbf v:=(v_1,v_2,v_3)^t\in\mathbf R^3$ and angular velocity vector $\boldsymbol\Omega:=(\Omega_1,\Omega_2,\Omega_3)^t\in\mathbf R^3$ (both in $\mathfrak e$) by respectively $\mathbf v:=R^t\dot{\mathbf r}$ and $\hat{\boldsymbol\Omega}:= R^t\dot R$, where for every vector $\mathbf u:=(u_1,u_2,u_3)^t\in\mathbf R^3$, $\hat{\mathbf u}$ is the unique skew-symmetric matrix satisfying $\hat{\mathbf u}x:=\mathbf u\times x$ for every $x\in\mathbf R^3$. Vectors of $\mathbf R^6$ will sometimes be decomposed in the form $\mathbf f:=(\mathbf f^1,\mathbf f^2)^t\in\mathbf R^3\times\mathbf R^3$. For every $\mathbf f:=(\mathbf f^1,\mathbf f^2)\in\mathbf R^6$ and $\mathbf g:=(\mathbf g^1,\mathbf g^2)\in\mathbf R^6$, we can define $\mathbf f\star\mathbf g:=(\mathbf f^1\times\mathbf f^2,\mathbf f^1\times\mathbf g^2-\mathbf g^1\times\mathbf f^2)^t\in\mathbf R^6$.
\subsubsection*{Shape Changes}
Unless otherwise indicated, from now on all of the quantities will be expressed in the moving frame $\mathfrak{e}$. In our modeling, the domains occupied by the swimmer are images of the closed unit ball $\bar B$ by $C^1$ diffeomorphisms, isotopic the the identity, and tending to the identity at infinity, i.e. having the form ${\rm Id}+\vartheta$ where $\vartheta$ belongs to $D^1_0(\mathbf R^3)$ (definitions of the function spaces are given in the appendix, Section~\ref{SEC:diffeo}). With these settings, the shape changes over a time interval $[0,T]$ can be simply prescribed by means of an absolutely continuous function $t\in[0,T]\mapsto \vartheta_t\in D^1_0(\mathbf R^3)$. Then, denoting $\varTheta_t={\rm Id}+\vartheta_t$, the domain occupied by the swimmer at the time $t\geq 0$  is the closed, bounded, connected set $\bar{\mathcal B}_t:=\varTheta_t(\bar B)$ (do not forget that we are working in the frame $\mathfrak e$). We still require some notation: the unit ball's boundary is $\Sigma:=\partial B$, $\Sigma_t:=\varTheta_t(\Sigma)$ stands for the body-fluid interface, $\mathbf n_t$ is the unit normal vector to $\Sigma_t$ directed toward the interior of $\mathcal B_t$ and the fluid fills the exterior open set $\mathcal F_t:=\mathbf R^3\setminus\bar{\mathcal B}_t$. 

So-called {\it self-propelled} constraints are necessary to ensure that the deformations result from the work of internal forces (they avoid for instance translations to be considered as allowable shape changes).  Let a  function $\varrho\in C^0(\bar B)^+$ be given. The density of  the deformed body at the instant $t$, denoted by $\varrho_t\in C^0(\bar{\mathcal B}_t)$, is defined by:
\begin{equation}
\label{conservation_principle}
\varrho_t(x):=\varrho(\varTheta_t^{-1}(x))/J_t(\varTheta_t^{-1}(x)),\qquad(x\in \bar{\mathcal B}_t,\quad t\geq0),
\end{equation}
where $J_t:=|\det\nabla\varTheta_t|=\det\nabla\varTheta_t$ (we can drop the absolute values here because $\varTheta_t(x)\to x$ as $\|x\|_{\mathbf R^3}\to +\infty$ and $\det\nabla\varTheta_t(x)\neq 0$ for all $x\in\mathbf R^3$ and $t\geq0$).
The self-propelled constraints read:
\begin{subequations}
\label{self-propelled-cond}
\begin{equation}
\int_{\mathcal B_t}\!\!\varrho_t(x)x\,{\rm d}x=\mathbf 0\quad\text{and}\quad
\int_{\mathcal B_t}\!\!\varrho_t(x)\partial_t\varTheta_t(\varTheta_t^{-1}(x))\times x\,{\rm d}x=\mathbf 0\quad(t\geq 0).
\end{equation}
The former identity means that, as already mentioned before, the center of mass of the swimmer lies at any time at the origin of the moving frame. The latter relation tells us that the angular momentum (in $\mathfrak e$) has to remain constant as the swimmer undergoes shape changes. Equivalent formulations can be obtained up to a change of variables:
\begin{equation}
\int_{B}\!\!\varrho(x)\varTheta_t(x)\,{\rm d}x=\mathbf 0\quad\text{and}\quad
\int_{B}\!\!\varrho(x)\partial_t\varTheta_t(x)\times \varTheta_t(x)\,{\rm d}x=\mathbf 0\quad(t\geq 0).
\end{equation}
\end{subequations}
 \begin{figure} \label{FIG_Kinematic}
 \centerline{\input{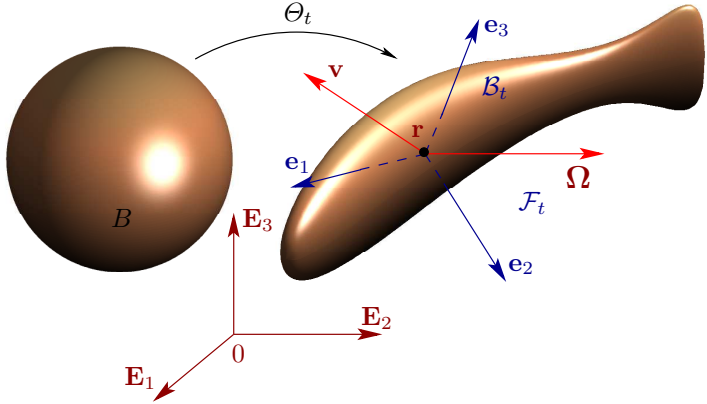}}
 \caption{Kinematics of the model: The Galilean frame $\mathfrak E:=(\mathbf E_j)_{1\leq j\leq 3}$ and the
 moving frame $\mathfrak e:=(\mathbf e_j)_{1\leq j\leq 3}$ with $\mathbf e_j=R\,\mathbf E_j$
 ($R\in{\rm SO}(3)$). Quantities are usually expressed in
 the moving frame. The domain of the body is $\bar{\mathcal B}_t$ at the time $t$ and $\mathcal B_t$ is the image of the
 unit ball $B$ by a diffeomorphism $\varTheta_t$. The open set $\mathcal F_t:=\mathbf R^3\setminus\bar{\mathcal B}_t$ is the domain of the fluid. The center
 of mass of the body is denoted $\mathbf r$ (in $\mathfrak E$), $\mathbf v:=R^t\dot{\mathbf r}$ is its translational velocity (in $\mathfrak e$) and $\boldsymbol\Omega$ its angular velocity.}
 \end{figure}
\subsubsection*{The Flow}
The fluid is assumed to be inviscid and incompressible. We denote by $\varrho_f>0$ its constant density. The flow is governed by Euler equations. According to Helmholtz's third theorem, if the flow is irrotational at the initial time, it remains irrotational for all times. In this case, the Eulerian velocity is equal at any time to the gradient of a potential function. According to Kirchhoff's law, the potential can be decomposed into a linear combination of elementary potentials, each one connected to a degree of freedom of the system (they consist here in the $6$ degrees of freedom of the rigid motion of the body plus those connecting to the deformations). 
These ideas have been thoroughly described in a series of papers \cite{Kanso:2005aa,Munnier:2008aa, Munnier:2008ab, Chambrion:2010aa}, to which we refer for further explanations.

The elementary potentials $\psi^i_t$, $(i=1,\ldots,6)$ corresponding to the rigid motion of the swimmer are harmonic in $\mathcal F_t$, tend to zero as infinity and satisfy the Neumann boundary conditions $\partial_n\psi^i_t=(\mathbf e_i\times x)\cdot \mathbf n_t$ ($i=1,2,3$) and $\partial_n\psi^i_t=\mathbf e_{i-3}\cdot\mathbf n_t$ ($i=4,5,6$)  on $\Sigma_t$.
They are well defined in the weighted Sobolev space $W^1(\mathcal F_t)$ (defined in the appendix, Section~\ref{SEC:diffeo}; see also \cite{Amrouche:1997aa} for details). The overall potential connecting to the rigid displacement is $\psi_t:=\sum_{i=1}^3\Omega_i\psi^i_t+\sum_{i=4}^5v_{i-3}\psi^i_t$ ($t\geq 0$).
On the other hand, the elementary potential $\varphi_t$ associated to the shape changes, harmonic as well in $\mathcal F_t$, satisfies the boundary condition $\partial_n\varphi_t=\mathbf w_t\cdot \mathbf n_t$ on $\Sigma_t$ ($i=1,\ldots,n$),
where  $\mathbf w_t(x):=\partial_t\varTheta_t(\varTheta^{-1}_t(x))$ $(x\in\mathbf R^3)$. Like the functions $\psi^i_t$ ($i=1,2,3$), $\varphi_t$ belongs to $W^1(\mathcal F_t)$ for all $t>0$.

\subsubsection*{Dynamics}
The modeling of moving rigid (or shape changing) bodies in an ideal fluid classically involves the notion of mass matrices. The mass of the body is $m:=\int_B\varrho\,{\rm d}x$ and its inertia tensor at the time $t\geq 0$ is defined by:
\begin{equation}
\label{def:Is}
\mathbb I(t):=\int_{\mathcal B_t}\varrho_t\big[\|x\|_{\mathbf R^3}^2{\rm Id}-x\otimes x\big]\,{\rm d}x=\int_{B}\varrho\big[\|\varTheta_t(x)\|^2_{\mathbf R^3}{\rm Id}-\varTheta_t(x)\otimes \varTheta_t(x)\big]\,{\rm d}x.
\end{equation}
We introduce $\mathbb M^r_b(t):={\rm diag}(\mathbb I(t),m{\rm I d})$ (a $6\times 6$ symmetric bloc diagonal matrix), $\mathbb M^r_f(t)$ (a $6\times 6$ symmetric matrix as well) whose entries read:
\begin{equation}
\label{def:coeff_mrf}
\varrho_f\int_{\mathcal F_t}\nabla\psi^i_t\cdot\nabla\psi^j_t\,{\rm d}x,\quad(1\leq i,j\leq 6),
\end{equation}
and we denote $\mathbb M^r(t):=\mathbb M^r_b(t)+\mathbb M^r_f(t)$.
We also need the $6\times 1$ column vector $\mathbf N(t)$, homogeneous to a momentum, whose elements read:
\begin{equation}
\label{def:coeff_n}
\varrho_f\int_{\mathcal F_t}\nabla\psi^i_t\cdot\nabla\varphi_t\,{\rm d}x,\quad(1\leq i\leq 6).
\end{equation}
If we neglect the buoyancy force, it has been proved in a series of papers (we refer for instance to the already mentioned articles \cite{Kanso:2005aa, Munnier:2008ab} or \cite{Chambrion:2010aa}) that the swimming motion is governed by the equation:
\begin{subequations}
\label{main_dynamics}
\begin{equation}
\label{dynamics}
\begin{pmatrix}\boldsymbol\Omega\\
\mathbf v\end{pmatrix}=-\mathbb M^r(t)^{-1}\mathbf N(t),\qquad (t\geq 0).
\end{equation}
At this point, we can identified $\varrho_t$ (or more simply $\varrho$ since they are linked through relation \eqref{conservation_principle}) as a parameter and the control as being the function $t\in[0,T]\mapsto \vartheta_t\in D_0^1(\mathbf R^3)$. Notice that the dependance of the dynamics in the control is strongly nonlinear. Indeed $\vartheta_t$ describes the shape of the body and hence also the domain of the fluid in which are set the PDEs of the potential functions involved in the expressions of the mass matrices $\mathbb M^r(t)$ and $\mathbf N(t)$. 

To determine the rigid motion, Equation \eqref{dynamics} has to be supplemented with the ODE:
\begin{equation}
\label{complement}
\frac{d}{dt}\begin{pmatrix}
R\\
{\mathbf r}
\end{pmatrix}=
\begin{pmatrix}
R\,\hat{\boldsymbol\Omega}\\
R\,\mathbf v
\end{pmatrix},\qquad(t>0),
\end{equation}
\end{subequations}
together with Cauchy data for $R(0)$ and $\mathbf r(0)$. 
Remark that the dynamics does not depend on $\varrho$ and $\varrho_f$ independently but only on the relative density $\varrho/\varrho_f$. So we can assume, without loss of generality, that $\varrho_f=1$ in the sequel.
\subsection{Main results}
The first result ensures the well posedness of System \eqref{main_dynamics} and the continuity of the input-output mapping:

\begin{proposition}
\label{existence}
For any $T>0$, any $\varrho\in C^0(\bar B)^+$, any absolutely continuous function $t\in[0,T]\mapsto \vartheta_t\in D^1_0(\mathbf R^3)$ (respectively of class $C^p$, $p=1,\ldots,+\infty$ or analytic) and any inital data $(R(0),\mathbf r(0))\in{\rm SO}(3)\times\mathbf R^3$, System  \eqref{main_dynamics} admits a unique solution $t\in[0,T]\mapsto (R(t),\mathbf r(t))\in{\rm SO}(3)\times\mathbf R^3$ (in the sense of Carath\'eodory) absolutely continuous on $[0,T]$ (respectively of class $C^p$ or analytic).

Let $t\in[0,T]\mapsto \vartheta^j_t\in D^1_0(\mathbf R^3)$ (for $j=1,\ldots,+\infty$) be a sequence of controls in $AC([0,T],D^1_0(\mathbf R^3))$ (see Section~\ref{SEC:diffeo} for a definition of this space) which converges in this space to a function $t\in[0,T]\mapsto\bar\vartheta_t\in D^1_0(\mathbf R^3)$. 
Let a pair $(R_0,\mathbf r_0)\in{\rm SO}(3)\times\mathbf R^3$ be given and denote $t\in[0,T]\mapsto(\bar R(t),\bar{\mathbf r}(t))\in{\rm SO}(3)\times\mathbf R^3$ the solution in $AC([0,T],{\rm SO}(3)\times\mathbf R^3)$ to System \eqref{main_dynamics}  
with control $\bar\vartheta$ and Cauchy data $(R_0,\mathbf r_0)$. Then, the unique solution $(R^j,\mathbf r^j)$ to System \eqref{main_dynamics} with control $\vartheta^j$ and Cauchy data $(R_0,\mathbf r_0)$ converges in $AC([0,T],{\rm SO}(3)\times\mathbf R^3)$ to $(\bar R,\bar{\mathbf r})$ as $j\to+\infty$. 
\end{proposition}

We denote by ${\rm M}(3)$ the Banach space of the $3\times 3$ matrices endowed wit any matrix norm. The main result of this article addresses the controllability of System \eqref{main_dynamics}:
\begin{theorem}(Synchronized Swimming)
\label{main_theorem_cont}
Assume that the following data are given: (i) A function $\bar\varrho$ in $C^0(\bar B)^+$ (the target density of the swimmer); (ii) A $C^1$ function $t\in[0,T]\mapsto \bar\vartheta_t\in D^1_0(\mathbf R^3)$ (the target shape changes) such that the pair $(\bar\varrho,\bar\vartheta)$ satisfies the self-propelled constraints \eqref{self-propelled-cond};
(iii) A $C^1$ function $t\in[0,T]\mapsto (\bar R(t),\bar{\mathbf r}(t))\in{\rm SO}(3)\times\mathbf R^3$ (the target trajectory to be followed). 
Then, for any $\varepsilon>0$, there exists a function  $\varrho\in C^0(\bar B)^+$ (the actual density of the swimmer) and a function $t\in[0,T]\mapsto \vartheta_t\in D^1_0(\mathbf R^3)$ (the actual shape changes that can be chosen of class $C^p$ for any $p=1,\ldots,+\infty$ or even analytic) such that the pair $(\varrho,\vartheta)$ satisfies \eqref{self-propelled-cond}, $\|\bar\varrho-\varrho\|_{C^0(\bar B)}<\varepsilon$, $\vartheta_0=\bar\vartheta_0$, $\vartheta_T=\bar\vartheta_T$ and $\sup_{t\in [0,T]}\Big(\|\bar\vartheta_{t}-\vartheta_t\|_{C^1_0(\mathbf R^3)^3}+\|\bar R(t)-R(t)\|_{{\rm M}(3)}+\|\bar{\mathbf r}(t)-\mathbf r(t)\|_{\mathbf R^3}\big)<\varepsilon$
where the function $t\in[0,T]\mapsto (R(t),\mathbf r(t))\in{\rm SO}(3)\times\mathbf R^3$ is the unique solution to system \eqref{main_dynamics} with initial data $(R(0),\mathbf r(0))=(\bar R(0),\bar{\mathbf r}(0))$ and control $\vartheta$.
\end{theorem}

This theorem tells us that, maybe up to an arbitrarily small change of its density, any weighted 3D body undergoing approximately any prescribed shape changes can approximately track by swimming any given trajectory.  It may seem surprising that the shape changes, which are supposed to be the control of our problem, can also be prescribed. Actually, be aware that they are only {\it approximately} prescribed. We are going to show precisely that arbitrarily small superimposed shape changes suffice for controlling the swimming motion. This result improves what has been done in the article \cite{Chambrion:2010aa}  for a particular 2D model. 

When no macro shape changes are prescribed we have:
\begin{theorem}(Freestyle Swimming)
\label{main_theorem_free}
Assume that the following data are given: (i) A pair $(\bar\varrho,\bar\vartheta)\in C^0(\bar B)^+\times D^1_0(\mathbf R^3)$ satisfying \eqref{self-propelled-cond} (the target density and shape at rest) (ii) A $C^1$ function $t\in[0,T]\mapsto (\bar R(t),\bar{\mathbf r}(t))\in{\rm SO}(3)\times\mathbf R^3$ (the target trajectory). 
Then, for any $\varepsilon>0$ there exists a pair $(\varrho,\vartheta)\in C^0(\bar B)^+\times D^1_0(\mathbf R^3)$ (the actual density and shape at rest) such that (i) $\int_B\varrho\,\varTheta\,{\rm d}x=\mathbf 0$  (where $\varTheta:={\rm Id}+\vartheta$) (ii) $\|\bar\varrho-\varrho\|_{C^0(\bar B)}+\|\bar\vartheta-\vartheta\|_{D^1_0(\mathbf R^3)}<\varepsilon$ and (iii) for almost any $5$-uplet $(\mathbf V_1,\ldots,\mathbf V_5)\in (C^1_0(\mathbf R^3)^3)^5$ satisfying $\int_B \varrho\,\mathbf V_i\,{\rm d}x=\mathbf 0$, $\int_B\varrho\,\mathbf \varTheta\times\mathbf V_i\,{\rm d}x=\mathbf 0$ and $\int_B\varrho\,\mathbf V_i\times\mathbf V_j\,{\rm d}x=\mathbf 0$ ($i,j=1,\ldots,5$), there exists a function  $t\in[0,T]\mapsto s(t)\in\mathbf R^5$ (that can be chosen of class $C^p$ for any $p=1,\ldots,+\infty$ or even analytic) such that, using  $\vartheta_t:=\vartheta+\sum_{i=1}^5s_i(t)\mathbf V_i\in D^1_0(\mathbf R^3)$ as control in the dynamics \eqref{main_dynamics}, we get $\sup_{t\in [0,T]}\big(\|\bar R(t)-R(t)\|_{{\rm M}(3)}+\|\bar{\mathbf r}(t)-\mathbf r(t)\|_{\mathbf R^3}\big)<\varepsilon$
where the function $t\in[0,T]\mapsto (R(t),\mathbf r(t))\in{\rm SO}(3)\times\mathbf R^3$ is the unique solution to ODEs \eqref{main_dynamics} with initial data $(R(0),\mathbf r(0))=(\bar R(0),\bar{\mathbf r}(0))$.
\end{theorem}

Differently stated, we claim in this Theorem that any weighted 3D body (maybe up to an arbitrarily small change of its density) is able to swim by means of allowable deformations (i.e. satisfying the self-propelled constraints) obtained as a suitable combination of  pretty much any given five basic movements. 

The proofs rely on the following main ideas: First, we shall identify a set of parameters necessary to thoroughly characterize a swimmer and its way of swimming (these parameters are its density, its shape and a finite number of basic movements, satisfying the self-propelled constraints \eqref{self-propelled-cond}). Any set of such parameters will be termed a {\it swimmer configuration} (denoted SC in short). Then, the set of all of the SC will be shown to be an (infinite dimensional) analytic connected  embedded submanifold of a  Banach space. 

The second step of the reasoning will consist in proving that the swimmer's ability to track any given trajectory (while undergoing any given shape changes) is related to the vanishing of some analytic functions depending on the SC. These functions are connected to the determinant of some vector fields and their Lie brackets (we will use here some classical result of Geometric Control Theory). Eventually, by direct calculation, we will prove that at least one swimmer (corresponding to one particular SC) has this ability. An elementary property of analytic functions will eventually allow us to conclude that almost any SC (or equivalently any swimmer) has this property.

\begin{remark}
The authors conjecture that in both Theorem~\ref{main_theorem_cont} and Theorem~\ref{main_theorem_free}, the actual density $\varrho$ can be chosen equal to the target density $\bar\varrho$. At this point however and although it is very unlikely, it can not be excluded that all of the swimmers with a particular density might be unable to swim. This issue also appeared in \cite{Chambrion:2010aa}.
\end{remark}
\subsection{Outline of the paper}
The next Section is dedicated to the notion of swimmer configuration (definition and properties). In Section~\ref{sensiti:mass} we show that the mass matrices are analytic functions in the SC (seen as a variable) and in Section~\ref{SEC:control_problem} we will restate the control problem in order to fit with the general framework of Geometric Control Theory. A particular case of swimmer will be shown to be controllable. In Section~\ref{sec:main_results} the proof of the main results will be performed. Section~\ref{sec:conclusion} contains some words of conclusion. Many technical results are gathered in the appendix to avoid overloading the rest of the paper.  
\section{Swimmer Configuration}
\label{section:SC}
A {\it swimmer configuration} is a set of parameters characterizing swimmers whose deformations consist in a combination of a finite number of basic movements. 
\begin{definition}
\label{def:sc}
For any positive integer $n$, we denote $\mathcal C(n)$ the subset of $C^0(\bar B)^+\times D^1_0(\mathbf R^3)\times (C^1_0(\mathbf R^3)^3)^n$ consisting of all of the triplets $c:=(\varrho,\vartheta,\mathcal V)$ such that, 
denoting $\varTheta:={\rm Id}+\vartheta$ and $\mathcal V:=(\mathbf V_1,\ldots,\mathbf V_n)$, the following conditions hold (i) the set $\{\mathbf V_i|_{\bar B}\cdot\mathbf e_k,\,1\leq i\leq n,\,k=1,2,3\}$ is a free family in $C^1_0(\bar B)$ (ii) every pair $(\mathbf V,\mathbf V')$ of elements of $\{\varTheta,\mathbf V_1,\ldots,\mathbf V_n\}$ satisfies
$\int_{B}\varrho\,\mathbf V\,{\rm d}x=\mathbf 0$ and $\int_{B}\varrho\,\mathbf V\times \mathbf V'\,{\rm d}x=\mathbf 0$.

We call swimmer configuration (SC in short) any element $c$ of $\mathcal C(n)$.
\end{definition}

By definition, $D^1_0(\mathbf R^3)$ is open in $C^1_0(\mathbf R^3)^3$ (see appendix, Section~\ref{SEC:diffeo}). We deduce that for any $c\in\mathcal C(n)$, the set $\{s:=(s_1,\ldots,s_n)^t\in\mathbf R^n\,:\,\vartheta+\sum_{i=1}^n s_i\mathbf V_i\in D^1_0(\mathbf R^3)\}$ is open as well in $\mathbf R^n$ and we denote $\mathcal S(c)$ its connected component containing $s=0$.
\begin{definition}
For any positive integer $n$, we call extended swimmer configuration (ESC in short) any pair $\mathbf c:=(c,s)$ such that $c\in\mathcal C(n)$ and $s\in\mathcal S(c)$. We denote $\mathcal C_X(n)$ the set of all of these pairs.
\end{definition}
\subsubsection*{Restatement of the problem in terms of SC and ESC}
Pick a SC in $\mathcal C(n)$ (for some integer $n$). The characteristics of the corresponding swimmer can be deduced from $c$ as follows:  If $c$ is equal to $(\varrho,\mathcal\vartheta,\mathcal V)$ with $\mathcal V:=(\mathbf V_1,\ldots,\mathbf V_n)$, the shape of the swimmer at rest is $\bar{\mathcal B} :=\varTheta(\bar B)$ where $\varTheta:={\rm Id}+\vartheta$. When swimming, it can occupy the domains $\bar{\mathcal B}_{\mathbf c}:=\varTheta_s(\bar B)$ for all $s\in\mathcal S(c)$ ($\mathbf c:=(c,s)\in \mathcal C(n)$ is hence an ESC), where $\varTheta_s:={\rm Id}+\vartheta+\sum_{i=1}^n s_i\mathbf V_i$.
Still for any $s\in\mathcal S(c)$, its density is the function $\varrho_s\in C^0(\bar{\mathcal B}_{\mathbf c})^+$ defined by $\varrho_s(x):=\varrho(\varTheta^{-1}_s(x))/J_s(\varTheta_s^{-1}(x))$ with $J_s:=\det\nabla\varTheta_s$. Notice that within this construction, the shape changes on a time interval $[0,T]$ ($T>0$) are merely given through an absolutely continuous  function $t:[0,T]\mapsto s(t)\in\mathcal S(c)$. 
If $t\in[0,T]\mapsto \dot s(t)\in\mathbf R^n$ stands for its time derivative, the Lagrangian velocity at a point $x$ of $\bar B$ is $\sum_{i=1}^n\dot s_i(t) \mathbf V_i(x)$ while the Eulerian velocity at a point $x\in\bar{\mathcal B}_{\mathbf c}$ is $\sum_{i=1}^n\dot s_i(t) \mathbf w^i_s(x)$ where $\mathbf w^i_s(x):=\mathbf V_i(\varTheta_s^{-1}(x))$. Due to assumption~(ii) of Definition~\ref{def:sc}, the self-propelled constraints \eqref{self-propelled-cond} are automatically satisfied.

The harmonic elementary potential functions of the fluid corresponding to the rigid motions depend only on the ESC. Therefore, they will be denoted in the sequel $\psi^i_{\mathbf c}$ instead of $\psi^i_t$. The same remark holds for the inertia tensor $\mathbb I(t)$ and the mass matrices $\mathbb M^r(t)$, $\mathbb M^r_b(t)$ and $\mathbb M^r_f(t)$ whose notation is turned into $\mathbb I(\mathbf c)$, $\mathbb M^r(\mathbf c)$, $\mathbb M^r_b(\mathbf c)$ and $\mathbb M^r_f(\mathbf c)$ respectively. The elementary potential connected to the shape changes can be decomposed into $\sum_{i=1}^n\dot s_i\varphi^i_{\mathbf c}$. In this sum, each potential function $\varphi^i_{\mathbf c}$ is harmonic in $\mathcal F_{\mathbf c}:=\mathbf R^3\setminus \bar{\mathcal B}_{\mathbf c}$ and satisfies on $\Sigma_{\mathbf c}:=\partial {\mathcal B}_{\mathbf c}$ the Neumann boundary conditions $\partial_n\varphi^i_{\mathbf c}=\mathbf w^i_s\cdot \mathbf n_{\mathbf c}$
, $\mathbf n_{\mathbf c}$ being the unit normal to $\Sigma_{\mathbf c}$ directed toward the interior of $\mathcal B_{\mathbf c}$.
Introducing the mass matrix $\mathbb N(\mathbf c)$, whose elements are $\varrho_f\int_{\mathcal B_{\mathbf c}}\nabla\psi^i_{\mathbf c}\cdot\nabla\varphi^j_{\mathbf c}{\rm d}x$ ($1\leq i\leq 6,\,1\leq j\leq n$) 
(recall that $\varrho_f$ can be chosen equal to 1), the dynamics \eqref{dynamics} can now be rewritten in the form:
\begin{equation}
\label{dynamics:1}
\begin{pmatrix}\boldsymbol\Omega\\
\mathbf v\end{pmatrix}=-\mathbb M^r(\mathbf c)^{-1}\mathbb N(\mathbf c)\dot s,\qquad (t\geq 0).
\end{equation}
Let us focus on the properties of $\mathcal C(n)$ and $\mathcal C_X(n)$.

\begin{theorem}
\label{theorem:1}
For any positive integer $n$, the set  ${\mathcal C}(n)$ is an analytic connected embedded submanifold  of $C^0(\bar B)\times C^1_0(\mathbf R^3)^3\times (C^1_0(\mathbf R^3)^3)^n$ of codimension $N:=3(n+2)(n+1)/2$.
\end{theorem}

The definition and the main properties of analytic functions in Banach spaces are summarized in the article \cite{Whittlesey:1965aa}. 

\begin{proof}
For any $c:=(\varrho,\vartheta,\mathcal V)\in C^0(\bar B)\times C^1_0(\mathbf R^3)^3\times (C^1_0(\mathbf R^3)^3)^n$, denote $\mathbf V_0:={\rm Id}+\vartheta$ and $\mathcal V:=(\mathbf V_1,\ldots,\mathbf V_n)$. Then, define 
for $k=0,1,\ldots,n$, the functions $\Lambda_k:C^0(\bar B)\times C^1_0(\mathbf R^3)^3\times (C^1_0(\mathbf R^3)^3)^n\to\mathbf R^{3(n+1-k)} $ by $\Lambda_k(c):=
\Big(\int_{B}\varrho\,\mathbf V_k\,{\rm d}x,\,
\int_{B}\varrho \,\mathbf V_k\times \mathbf V_{k+1}\,{\rm d}x,\,\ldots,
\int_{B}\varrho\,\mathbf V_k\times \mathbf V_n\,{\rm d}x\Big) ^t$.
Every function $\Lambda_k$ is analytic and we draw the same conclusion for $\Lambda:=(\Lambda_0,\ldots,\Lambda_n)^t:C^0(\bar B)\times C^1_0(\mathbf R^3)^3\times (C^1_0(\mathbf R^3)^3)^n\to \mathbf R^N$ ($N:=3(n+2)(n+1)/2$). In order to prove that $\partial_c\Lambda(c)$ (the differential of $\Lambda$ at the point $c$) is onto for any $c\in\mathcal C(n)$, assume that there exist $(n+2)(n+1)/2$ vectors $\boldsymbol\alpha_i^j\in\mathbf R^3$ ($0\leq i\leq j\leq n$) such that:
\begin{equation}
\label{gamma_onto}
\sum_{i=0}^n\boldsymbol\alpha_i\cdot\langle\partial_c\Lambda(c),c^h\rangle=\mathbf 0,\qquad\forall\,c^h\in C^0(\bar B)\times C^1_0(\mathbf R^3)^3\times (C^1_0(\mathbf R^3))^3,
\end{equation}
where $\boldsymbol\alpha_i:=(\boldsymbol\alpha_i^i,\boldsymbol\alpha_i^{i+1},\ldots,\boldsymbol\alpha_i^n)^t\in\mathbf R^{3(n+1-i)}$ ($j=0,\ldots,n$) and $c^h:=(\varrho^h,\vartheta^h,\mathcal V^h)\in C^0(\bar B)\times C^1_0(\mathbf R^3)^3\times (C^1_0(\mathbf R^3)^3)^n$ with $\mathbf V_0^h:={\rm Id}+\vartheta^h$ and $\mathcal V^h:=(\mathbf V^h_1,\ldots,\mathbf V^h_n)$.
Reorganizing the terms in \eqref{gamma_onto}, we obtain that:
\begin{multline*}
\int_{B}\varrho^h\Big[\sum_{k=0}^n\boldsymbol\alpha_k^k\cdot\mathbf V_k+\sum_{0\leq i<j\leq n}\!\!\!\boldsymbol\alpha_i^j\cdot(\mathbf V_i\times\mathbf V_j)\Big]{\rm d}x+\\
\sum_{k=0}^n\int_{B}\varrho\mathbf V_k^h\cdot\Big[\sum_{j=0}^{k-1}\boldsymbol\alpha_j^k\times\mathbf V_j+\boldsymbol\alpha_k^k-\sum_{j=k+1}^{n}\boldsymbol\alpha_k^j\times\mathbf V_j\Big]{\rm d}x=0.
\end{multline*}
Since this identity has to be satisfied for any $(\varrho^h,\vartheta^h,\mathcal V^h)\in C^0(\bar B)\times C^1_0(\mathbf R^3)^3\times (C^1_0(\mathbf R^3))^3$, we deduce that, for every $k=0,\ldots,n$:
\begin{equation}
\label{identity:1}
 \sum_{p=0}^n\boldsymbol\alpha_p^p\cdot\mathbf V_p+\!\!\!\!\!\sum_{0\leq i<j\leq n}\!\!\!\boldsymbol\alpha_i^j\cdot(\mathbf V_i\times\mathbf V_j)=0
 \text{~~~and~~~}\sum_{j=0}^{k-1}\boldsymbol\alpha_j^k\times\mathbf V_j+\boldsymbol\alpha_k^k-\!\!\!\sum_{j=k+1}^{n}\boldsymbol\alpha_k^j\times\mathbf V_j=\mathbf 0.\end{equation}
Multiplying the latter equality by $\varrho$ and integrating over $B$, we get that $\boldsymbol\alpha^k_k=\mathbf 0$ ($k=0,\ldots,n$).  The first equality in \eqref{identity:1} is now just a linear combination of the second ones (for $k=0,\ldots,n$), so we can drop it. 
Taking into account Hypothesis (ii) of Definition~\ref{def:sc}, latter identity in \eqref{identity:1} with $k=0$ leads to $\boldsymbol\alpha_0^j=\mathbf 0$ for every $j=1,\ldots,n$. There are no more terms involving $\mathbf V_0$ in the other equations and invoking again Hypothesis (ii) we eventually get $\boldsymbol\alpha_i^j=\mathbf 0$ for $1\leq i<j\leq n$. So, equality \eqref{gamma_onto} entails that $\boldsymbol\alpha_i=\mathbf 0$ for all $i=0,\ldots,n$ and the mapping $\partial_c\Lambda(c)$ is onto for all $c\in\mathcal C(n)$.

The linear space $X=\mathrm{Ker}\,\partial_c \Lambda(c)$ is closed since $\Lambda$ is analytic. Let $Y$ be an algebraic supplement of $X$ in $C^0(\bar{B})\times C^1_0(\mathbf R^3)^3 \times (C^1_0(\mathbf R^3)^3)^n$, and denote by $\Pi_Y$ the linear projection onto $Y$ along $X$. A crucial observation is that the linear space $Y$ is isomorphic to $\mathbf{R}^N$ and hence it is finite dimensional and closed in  $C^0(\bar{B})\times C^1_0(\mathbf R^3)^3 \times (C^1_0(\mathbf R^3)^3)^n$. Define the analytic mapping $f:X\times Y \rightarrow \mathbf{R}^N$ by
$f(x,y)=\Lambda(c+x+y)$. The mapping $\partial_yf(0,0)=\partial_c\Lambda(c)\circ \Pi_Y$ being onto, the implicit function theorem (analytic version in Banach spaces, see \cite{Whittlesey:1965aa}) asserts that there exist an open neighborhood $\mathcal O_1$ of $0$ in $X$, an open neighborhood $\mathcal O_2$ of $0$ in $Y$, and an analytic mapping $g:\mathcal O_1 \rightarrow Y$ such that $g(0)=0$ and, for every $(x,y)$ in $\mathcal O_1\times \mathcal O_2$, the two following assertions are equivalent: (i) $f(x,y)=0$ (or, in other words, $c+x+y$ belongs to $\mathcal{C}(n)$), and
(ii) $y=g(x)$.
The analytic mapping $g$ provides a local parameterization of $\mathcal{C}(n)$ in a neighborhood of $c$.

In order to prove that $\mathcal C(n)$ is path-connected, consider two elements $c^\dagger:=(\varrho^\dagger,\vartheta^\dagger,\mathcal V^\dagger)$ and $c^\ddagger:=(\varrho^\ddagger,\vartheta^\ddagger,\mathcal V^\ddagger)$ of $\mathcal C(n)$ and denote $\varTheta^\dag:={\rm Id}+\vartheta^\dag$, $\mathcal V^\dag:=(\mathbf V_1^\dag,\ldots,\mathbf V_n^\dag)$ and $\varTheta^\ddag:={\rm Id}+\vartheta^\ddag$, $\mathcal V^\ddag:=(\mathbf V_1^\ddag,\ldots,\mathbf V_n^\ddag)$.
According to Definition~\ref{def_D10}, $D^1_0(\mathbf R^3)$ is open and connected. It entails that it is always possible to find a continuous, piecewise linear path $t:[0,1]\mapsto \bar\vartheta_t\in D^1_0(\mathbf R^3)$ such that $\bar\vartheta_{t=0}=\vartheta^\dag$ and $\bar\vartheta_{t=1}=\vartheta^\ddag$. We introduce $0=t_0<t_1<\ldots<t_k=1$, a subdivision of the interval $[0,1]$ such that $t\mapsto\bar\vartheta_t$ is linear on every subinterval $[t_j,t_{j+1}]$ ($j=0,\ldots,k-1$) and we denote $\bar\varTheta_t:={\rm Id}+\bar\vartheta_t$, $\bar\vartheta^j:=\bar\vartheta_{t=t_j}$, $\bar\varTheta^j:={\rm Id}+\bar\vartheta^j$  ($j=0,\ldots,k$). Let us introduce as well the continuous  functions $t\in[0,1]\mapsto \varrho_t:=t\varrho^\ddag+(1-t)\varrho^\dag\in C^0(\bar B)$ and $t\in[0,1]\mapsto \mathbf u(t):=-\int_B\varrho_t\bar\varTheta_t{\rm d}x/\int_B\varrho_t{\rm d}x\in\mathbf R^3$. The set $\cup_{t\in[0,T]}\bar\varTheta_t(\bar B)+\mathbf u(t)$ being compact, it is contained in a large ball $\Omega$. We introduce $\Omega'$ an even larger ball containing $\Omega$ and a cut-off function $\chi$ defined in $\mathbf R^3$ such that $0\leq \chi\leq 1$, $\chi=1$ in $\Omega$ and $\chi=0$ in $\mathbf R^3\setminus\bar\Omega'$. 
The derivative $\dot{\mathbf u}$ of $\mathbf u$ exists everywhere on $]0,T[$ excepted maybe at the points $t_1,\ldots,t_{k-1}$. The flow associated with the Carath\'eodory's solutions of the Cauchy problem: $\dot X_t(x)=\dot{\mathbf u}(t)\chi(X_t(x))$, $(t>0)$, $X_{t=0}(x)=x$ is well defined (see \cite[Theorem 1A, page 57]{Lee:1967aa}). Moreover, for every fixed $t\in[0,1]$, the mapping $x\in\mathbf R^3\mapsto X_t(x)\in\mathbf R^3$ is $C^\infty$. Consider now the mappings $t\in[0,1]\mapsto \vartheta_t:=X_t\circ\bar\varTheta_t-{\rm Id}$ and $\varTheta_t:={\rm Id}+\vartheta_t$. If $x\in \mathbf R^3\setminus\bar\Omega'$, $\varTheta_t(x)=\bar\varTheta_t(x)$ for all $t\in[0,T]$ and if $x\in\bar B$ then $\vartheta_t(x)=\bar\vartheta_t(x)+\mathbf u(t)$ and $\varTheta_t(x)=\bar\varTheta_t(x)+\mathbf u(t)$. Notice that $\vartheta_t\in D^1_0(\mathbf R^3)$ for all $t\in[0,T]$ and $\int_B\varrho_t(x)\varTheta_t(x){\rm d}x=0$ for all $t\in[0,T]$. Since $C^1_0(\mathbf R^3)^3$ is an infinite dimensional Banach space, it is always possible to find by induction $\mathbf W_1,\mathbf W_2,\ldots,\mathbf W_n$ in $C^1_0(\mathbf R^3)^3$ such that (i) both families 
$\{\mathbf W_1|_{\bar B}\cdot \mathbf e_k,\ldots,\mathbf W_n|_{\bar B}\cdot \mathbf e_k,\mathbf V_1^\dag|_{\bar B}\cdot \mathbf e_k,\ldots,\mathbf V_n^\dag|_{\bar B}\cdot \mathbf e_k,\,k=1,2,3\}$ and $\{\mathbf W_1|_{\bar B}\cdot \mathbf e_k,\ldots,\mathbf W_n|_{\bar B}\cdot \mathbf e_k,\mathbf V_1^\ddag|_{\bar B}\cdot \mathbf e_k,\ldots,\mathbf V_n^\ddag|_{\bar B}\cdot \mathbf e_k,\,k=1,2,3\}$ are free in $C^1_0(\mathbf R^3)$ and (ii) for any pair of elements $\mathbf V$, $\mathbf V'($ both picked in the same family, $\int_B\varrho^\dag\mathbf V{\rm d}x=\mathbf 0$, $\int_B\varrho^\ddag\mathbf V{\rm d}x=\mathbf 0$, $\int_B\varrho^\dag\bar\varTheta^j\times\mathbf V{\rm d}x=\mathbf 0$, $\int_B\varrho^\ddag\bar\varTheta^j\times\mathbf V{\rm d}x=\mathbf 0$  (for all $j=1,\ldots,k$), $\int_B\varrho^\dag\mathbf V\times\mathbf V'{\rm d}x=\mathbf 0$ and $\int_B\varrho^\ddag\mathbf V\times\mathbf V'{\rm d}x=\mathbf 0$. Define the function $t\in[0,1]\mapsto\mathbf V^i_t\in C^1_0(\mathbf R^3)^3$ by $\mathbf V^i_t:=(1-2t)\mathbf V_i^\dag+2t\mathbf W_i$ if $0\leq t\leq 1/2$ and $\mathbf V^i_t:=(2-2t)\mathbf W_i+(2t-1)\mathbf V^\ddag$ if $1/2<t\leq 1$ and denote $\mathcal V_t:=(\mathbf V^1_t,\ldots,\mathbf V^n_t)\in (C^1_0(\mathbf R^3)^3)^n$. Eventually, a continuous function linking $c^\dag$ to $c^\ddag$ is given by  $t\in[0,1]\mapsto c_t\in\mathcal C(n)$ with $c_t:=(\varrho^\dag,\vartheta^\dag,\mathcal V_{3t/2})$ if $0\leq t\leq 1/3$, $c_t:=(\varrho_{3t-1},\vartheta_{3t-1},\mathcal V_{1/2})$ if $1/3<t\leq 2/3$ and $c_t:=(\varrho^\ddag,\vartheta^\ddag,\mathcal V_{3t/2-1/2})$ if $2/3<t\leq 1$.
\end{proof}

We omit the proof of the following corollary, similar to that of the theorem above:

\begin{corollary}
\label{cor:1}
For any positive integer $n$, the set $\mathcal C_X(n)$ is an analytic connected embedded submanifold of $C^0(\bar B)\times C^1_0(\mathbf R^3)^3\times (C^1_0(\mathbf R^3)^3)^n\times \mathbf R^n$ of codimension $N:=3(n+2)(n+1)/2$.
\end{corollary}

We denote $\pi$ the projection of $\mathcal C(n)$ onto $C^0(\bar B)\times D^1_0(\mathbf R^3)$ defined by $\pi(c)=(\varrho,\vartheta)$ for all $c:=(\varrho,\vartheta,\mathcal V)\in\mathcal C(n)$. The proof of the following corollary is a straightforward consequence of arguments already used in the proof of Theorem~\ref{theorem:1}:

\begin{corollary}
\label{cor:2}
For any positive integer $n$ and for any $(\varrho,\vartheta)\in\pi(\mathcal C(n))$, the section $\pi^{-1}(\{(\varrho,\vartheta)\})$ is an embedded connected analytic submanifold of $\{\varrho\}\times\{\vartheta\}\times(C^1_0(\mathbf R^3)^3)^n$ (identified with $(C^1_0(\mathbf R^3)^3)^n$) of codimension $3n(n+3)/2$.
\end{corollary}

\section{Sensitivity Analysis of the Mass Matrices}
\label{sensiti:mass}
For any positive integers $k$ and $p$, we denote ${\rm M}(k,p)$ the vector space of the matrices of size $k\times p$ (or simply ${\rm M}(k)$ when $k=p$).
\begin{theorem}
\label{theorem:2}
For any positive integer $n$, the mappings $\mathbf c\in\mathcal C_X(n)\mapsto \mathbb M^r(\mathbf c)\in {\rm M}(6)$ and $\mathbf c\in\mathcal C_X(n)\mapsto \mathbb N(\mathbf c)\in {\rm M}(6,n)$ are analytic.
\end{theorem}

The method followed in this proof is inspired from \cite{Henrot:2005aa}. The result already appeared in \cite{Munnier:2008aa}, in a slightly different form though. Due to its crucial importance for our purpose, we recall here the main ideas. 

Let us begin with a preliminary lemma of which the statement requires introducing some material.
Thus, we denote $F:=\mathbf R^3\setminus \bar B$ (recall that $B$ is the unit ball, $\Sigma:=\partial B$ and $\mathbf n$ is the unit normal to $\Sigma$ directed toward the interior of $B$). For all  $\xi\in D^1_0(\mathbf R^3)$, we set $\varXi:={\rm Id}+\xi$, $\mathcal B_\xi:=\varXi(B)$, $\mathcal F_\xi:=\varXi(F)$, $\Sigma_\xi:=\varXi(\Sigma)$ and $\mathbf n_\xi$ stands for the unit normal to $\Sigma_\xi$ directed toward the interior of $\mathcal B_\xi$. We denote $\mathbf q:=(\xi,\mathcal W)$, with $\mathcal W:=(\mathbf W^1,\mathbf W^2)\subset (C^1_0(\mathbf R^3)^3)^2$, the elements of $\mathcal Q:=D^1_0(\mathbf R^3)\times(C^1_0(\mathbf R^3)^3)^2 $ and $\mathbf w^i_\xi:=\mathbf W^i(\varXi^{-1})$ ($i=1,2$). 
Finally, for every $\mathbf q\in\mathcal Q$, we define:
\begin{equation}
\label{def:M}
\varPhi(\mathbf q):=\int_{\mathcal F_\xi}\nabla\psi^1_{\mathbf q}(x)\cdot \nabla\psi^2_{\mathbf q}(x)\,{\rm d}x,
\end{equation}
where, for every $i=1,2$, the function $\psi^i_{\mathbf q}\in W^1(\mathcal F_\xi)$ (recall that the function spaces are defined in Section~\ref{SEC:diffeo}) is solution to the Laplace equation $-\Delta\psi^i_{\mathbf q}=0$ in $\mathcal F_\xi$  with Neumann boundary data $\partial_n\psi^i_{\mathbf q}=\mathbf w^i_\xi\cdot \mathbf n_\xi$ on $\Sigma_\xi$.
The solution has to be understood in the weak sense, namely:
\begin{equation}
\label{varia:1}
\int_{\mathcal F_\xi}\nabla\psi^i_{\mathbf q}(x)\cdot \nabla\varphi(x)\,{\rm d}x=\int_{\Sigma_\xi}(\mathbf w^i_\xi\cdot\mathbf n_\xi)(x)\varphi(x)\,{\rm d}\sigma,\quad\forall\,\varphi\in W^1(\mathcal F_\xi).
\end{equation}

\begin{lemma}
\label{lemma:import}
The mapping $\mathbf q\in\mathcal Q\mapsto \varPhi(\mathbf q)\in\mathbf R$ is analytic.
\end{lemma}

\begin{proof}
We pull back relation \eqref{varia:1} onto the domain $F$ using the diffeomorphism $\varXi$. We get:
$$
\int_{F}\nabla{\varPsi}^i_{\mathbf q}(x)\mathbb A_\xi(x)\cdot \nabla(\varphi\circ\varXi)(x)\,{\rm d}x=
\int_{\Sigma}(\mathbf W^i(x)\cdot \mathbf n_\xi(\varXi(x))\varphi(\varXi(x))J^\sigma_\xi(x)\,{\rm d}\sigma,\quad\forall\,\varphi\in W^1(\mathcal F_\xi),
$$
where $\varPsi^i_{\mathbf q}:=\psi^i_{\mathbf q}\circ\varXi$, $J_\xi:=\det(\nabla\varXi)$,
$\mathbb A_\xi:=(\nabla\varXi^t\nabla\varXi)^{-1}J_\xi$ and
$J^\sigma_\xi:=\|(\nabla\varXi^{-1})^t\mathbf n\|_{\mathbf R^3}J_\xi$ (usually referred to as the tangential Jacobian). 
In \eqref{varia:1}, if we specialize the test function to have the form $\varphi:=\phi\circ\varXi^{-1}$ with $\phi\in W^1(F)$, we obtain
$
\int_{F}{\varPsi}^i_{\mathbf q}(x)\mathbb A_\xi(x)\cdot \nabla\phi(x)\,{\rm d}x=
\int_{\Sigma}b^i_{\mathbf q}(x)\phi(x)\,{\rm d}\sigma$ for all $\phi\in W^1(F)$, 
where $b^i_{\mathbf q}:=(\mathbf W^i\cdot \mathbf n_\xi(\varXi))J^\sigma_\xi$ ($i=1,2$). We now claim that the mapping $\xi\in D^1_0(\mathbf R^3)\mapsto \mathbb A_\xi-{\rm Id} \in E^0_0(\Omega,{\rm M}(3))$ is analytic. Indeed, the mappings $\xi\in D^1_0(\mathbf R^3)\mapsto \nabla\varXi^t\nabla\varXi-{\rm Id}\in E^0_0(\Omega,{\rm M}(3))$, $A\in E^0_0(\Omega,{\rm M}(3))\mapsto ({\rm Id}+A)^{-1}-{\rm Id}\in E^0_0(\Omega,{\rm M}(3))$ and $\xi\in D^1_0(\mathbf R^3)\mapsto J_\xi-1\in C^0_0(\mathbf R^3)$  are analytic. Reasoning the same way, we can show that the mapping $\xi\in D^1_0(\mathbf R^3)\mapsto J^\sigma_\xi\in C^0(\Sigma)$ is analytic as well (notice that for all $\xi\in D^1_0(\mathbf R^3)$, the function $(\nabla\varXi^{-1})^t\mathbf n$ never vanishes on $\Sigma$). It is more complicated to prove that $\xi\in D^1_0(\mathbf R^3)\mapsto \mathbf n_\xi\circ\varXi\in C^0(\Sigma)^2$ is analytic, so we refer to \cite{Munnier:2008aa} for the details. This last result entails the analyticity of $\mathbf q\in\mathcal Q \mapsto b^i_{\mathbf q}\in C^0(\Sigma)$. Then, denoting by $W^1(F)'$ the dual space to $W^1(F)$, we consider the mapping
$\varGamma:(\mathbf q,u)\in\mathcal Q\times W^1(F)\mapsto \langle \mathbb A_\xi,u,\cdot\rangle-\langle b^i_{\mathbf q},\cdot\rangle\in W^1(F)'$, 
where
$\langle \mathbb A_\xi,u,\phi\rangle:=\int_{F}\nabla u(x)\mathbb A_\xi(x)\cdot \nabla\phi(x)\,{\rm d}x$ and 
$\langle b^i_{\mathbf q},\phi\rangle:=\int_{\Sigma}b^i_{\mathbf q}(x)\phi(x)\,{\rm d}\sigma$ ($\phi\in W^1(F)$).
We wish now to apply the implicit function theorem (analytic version in Banach spaces, as stated in \cite{Whittlesey:1965aa}) to the analytic function $\varGamma$. Observe, though, that we are only interested in the regularity result and not in the existence and uniqueness. Indeed, we already know  that for every $\mathbf q\in\mathcal Q$, there exists a unique function $\varPsi^i_{\mathbf q}\in W^1(F)$ such that $\varGamma(\mathbf q,\varPsi^i_{\mathbf q})=0$. The function $\varPsi^i_{\mathbf q}$ is equal to $\psi^i_{\mathbf q}\circ\varXi$ where $\psi^i_{\mathbf q}$ is the unique solution to the well posed variational problem \eqref{varia:1}.
The partial derivative $\partial_u \varGamma(\mathbf q,\varPsi^i_{\mathbf q})$ is defined by:
\begin{equation}
\label{riez}
\langle \partial_u \varGamma(\mathbf q,\varPsi^i_{\mathbf q}),u,\phi\rangle=\int_{F}\nabla u(x)\mathbb A_\xi(x)\cdot\nabla\phi(x)\,{\rm d}x,\quad\forall\phi\in W^1(F).
\end{equation}
For all $\xi\in D^1_0(\mathbf R^3)$, the matrix $\mathbb A_\xi$ is uniformly elliptic in $\mathbf R^3$ (there exists $\alpha_\xi>0$ such that $X^t\mathbb A_\xi(x) X>\alpha_\xi\|X\|_{\mathbf R^3}^2$ for all $X\in\mathbf R^3$ and all $x\in\mathbf R^3$). We deduce that the right hand side of \eqref{riez} can be chosen as the scalar product in $W^1(F)$ and hence that $\partial_u \varGamma(\mathbf q,\varPsi^i_{\mathbf q})$ is a continuous isomorphism from $W^1(F)$ onto its dual space according to the Riesz representation theorem. The implicit function theorem applies and asserts that the mapping $\mathbf q\in\mathcal Q\mapsto \varPsi^i_{\mathbf q}\in W^1(F)$ is analytic.

To conclude the proof, it remains only to observe that the function $\varPhi(\mathbf q)$ introduced in \eqref{def:M} can be rewritten, upon a change of variables as $\varPhi(\mathbf q)=\int_{F}\nabla\varPsi^1_{\mathbf q}(x)\mathbb A_\xi(x)\cdot\nabla\varPsi^2_{\mathbf q}(x)\,{\rm d}x$, 
which is indeed analytic as a composition of analytic functions.
\end{proof}

We can now give the proof of Theorem~\ref{theorem:2}.
\begin{proof}
Recall that the elements of the matrix $\mathbb M^r_f(\mathbf c)$ are defined in \eqref{def:coeff_mrf} and those of $\mathbb N(\mathbf c)$ in \eqref{def:coeff_n}. For any $\mathbf c:=(c,s)\in\mathcal C_X(n)$, where $c:=(\varrho,\vartheta,\mathcal V)$, we apply the lemma with $\xi:=\vartheta+\sum_{i=1}^ns_i\mathbf V_i$ and $\mathbf W^1,\mathbf W^2\in\{\mathbf e_i\times\varXi,\, \mathbf e_i,\,i=1,2,3\}$ to get that the mapping $\mathbf c\in\mathcal C_X(n)\mapsto \mathbb M^r_f(\mathbf c)\in {\rm M}(6)$ is analytic. 
To prove the analyticity of the elements of $\mathbb N(\mathbf c)$, we apply the lemma again with $\xi:=\vartheta+\sum_{i=1}^ns_i\mathbf V_i$, $\mathbf W^1\in \{ \mathbf e_i\times\varXi,\, \mathbf e_i,\,i=1,2,3\}$ and $\mathbf W^2\in\{\mathbf V_1,\ldots,\mathbf V_n\}$. Eventually, the analyticity of the elements of $\mathbb M^r_b(\mathbf c)$ is straightforward after rewriting the inertia tensor $\mathbb I(\mathbf c)$ defined in \eqref{def:Is} in the form (upon a change of variables) $\mathbb I(\mathbf c)=\int_{B}\varrho[\|\varXi\|_{\mathbf R^3}^2{\rm Id}-\varXi\otimes\varXi]\,{\rm d}x$, 
still with $\varXi:={\rm Id}+\xi$ and $\xi:=\vartheta+\sum_{i=1}^ns_i\mathbf V_i$.
\end{proof}

\section{Control}

\label{SEC:control_problem}
\subsection{Controllable SC}
Let us fix $c\in{\mathcal C}(n)$ (for some positive integer $n$) and recall that $\mathcal S(c)$ is the connected open subspace of $\mathbf R^n$ such that $(c,s)\in\mathcal C_X(n)$.  Introducing $(\mathbf f_1,\ldots,\mathbf f_n)$ an ordered orthonormal basis of $\mathbf R^n$, we can seek the function $t\in[0,T]\mapsto s(t)\in\mathcal S(c)$ as the solution of the ODE $\dot s(t)=\sum_{i=1}^n\lambda_i(t)\mathbf f_i$ where the functions $\lambda_i:t\in[0,T]\mapsto \lambda_i(t)\in\mathbf R$ are the new controls, and rewrite once more the dynamics \eqref{dynamics:1} as:
\begin{equation}
\label{dynamics:2}
\begin{pmatrix}\boldsymbol\Omega\\
\mathbf v\\
\dot s\end{pmatrix}=\sum_{i=1}^n\lambda_i(t)
\begin{pmatrix}-\mathbb M^r(c,s)^{-1}\mathbb N(c,s)\mathbf f_i\\
\mathbf f_i\end{pmatrix}, \qquad (t\geq 0).
\end{equation}
It is worth remarking that in this form, $s$ is no more the control but a variable which is meant to be controlled and $c\in\mathcal C(n)$ is a parameter of the dynamics. 
Considering \eqref{dynamics:2}, we are quite naturally led to introduce, for all $\mathbf c\in\mathcal C_X(n)$, the vector fields
$\mathbf X_i(\mathbf c):=-\mathbb M^r(\mathbf c)^{-1}\mathbb N(\mathbf c)\mathbf f_i\in\mathbf R^6$, $\mathbf Y_i(\mathbf c):=(\hat{\mathbf X}^1_i(\mathbf c)
,
{\mathbf X}^2_i(\mathbf c),
\mathbf f_i)^t\in T_{\rm Id}{\rm SO}(3)\times\mathbf R^3\times \mathbf R^n$ (we have used here the notation $\mathbf X_i:=(\mathbf X_i^1,\mathbf X_i^2)^t\in\mathbf R^3\times\mathbf R^3$) and 
$\mathbf Z_c^i(R,s):=\mathcal R_R\mathbf Y_i(\mathbf c)\in T_R{\rm SO}(3)\times\mathbf R^3\times \mathbf R^n$
where $\mathcal R_R:={\rm diag}(R,R,{\rm Id})\in{\rm SO}(6+n)$ is a bloc diagonal matrix.
The dynamics \eqref{dynamics:2} and the ODE \eqref{complement} can be gathered into a unique differential system:
\begin{equation}
\label{dynamics:3}
\frac{d}{dt}\begin{pmatrix}R\\
\mathbf r\\
s\end{pmatrix}=\sum_{i=1}^n\lambda_i(t)\mathbf Z^i_c(R,s),\qquad(t\geq 0).
\end{equation}
For every $i=1,\ldots,n$, the function $(R,\mathbf r,s)\in{\rm SO}(3)\times\mathbf R^3\times\mathcal S(c)\mapsto \mathbf Z^i_c(R,s)\in T_R{\rm SO}(3)\times\mathbf R^3\times \mathbf R^n$ can be seen as an analytic vector field (constant in $\mathbf r$) on the analytic connected manifold $\mathcal M(c):={\rm SO}(3)\times\mathbf R^3\times\mathcal S(c)$. We denote $\zeta$ any element $(R,\mathbf r,s)\in\mathcal M(c)$ and we define $\mathcal Z(c)$ as the family of vector fields $(\mathbf Z^i_c)_{1\leq i\leq n}$ on $\mathcal M(c)$.  
\begin{lemma}Let $c$ be a SC fixed in $\mathcal C(n)$ ($n$ a positive integer). 
If there exists $\zeta\in\mathcal M(c)$ such that ${\rm dim}\,{\rm Lie}_{\zeta}\mathcal Z(c)=6+n$, then the orbit of $\mathcal Z(c)$ through any $\zeta\in\mathcal M(c)$ is equal to the whole manifold $\mathcal M(c)$.
\end{lemma}

\begin{proof}
Rashevsky Chow Theorem (see \cite{Agrachev:2004aa}) applies: If ${\rm Lie}_{\zeta}\mathcal Z(c)=T_{\zeta}\mathcal M(c)$ for all $\zeta\in\mathcal M(c)$ (or more precisely, for all $(R,s)\in{\rm SO}(3)\times\mathcal S(c)$ since $\mathbf Z^i_c$ does not depend on $\mathbf r$) then the orbit of $\mathcal Z(c)$ through any point of $\mathcal M(c)$ is equal to the whole manifold. 
Let us compute the Lie bracket $[\mathbf Z^i_c(R,s),\, \mathbf Z^j_c(R,s)]$ for $1\leq i,j\leq n$ and $(R,s)\in{\rm SO}(3)\times\mathcal S(c)$. We get:
\begin{equation}
\label{identity:lie_brackets}
[\mathbf Z^i_c(R,s),\, \mathbf Z^j_c(R,s)]=\mathcal R_R\begin{pmatrix}
\widehat{(\mathbf X_i^1\times\mathbf X_j^1)}(\mathbf c)\\
({\mathbf X}^1_i\times\mathbf X_j^2-{\mathbf X}^1_j\times\mathbf X_i^2)(\mathbf c)\\
\mathbf 0
\end{pmatrix}
+\mathcal R_R\begin{pmatrix}
\widehat{(\partial_{s_i}{\mathbf X}^1_j-\partial_{s_j}{\mathbf X}_i^1)}(\mathbf c)\\
(\partial_{s_i}{\mathbf X}^2_j-\partial_{s_j}{\mathbf X}_i^2)(\mathbf c)\\
\mathbf 0
\end{pmatrix}.
\end{equation}
By induction, we can similarly prove that the Lie brackets of any order at any point $\zeta\in\mathcal M(c)$ have the same general form, namely the matrix $\mathcal R_R$ multiplied by an element of $T_{({\rm Id},\mathbf 0,s)}\mathcal M(c)$. We deduce that the dimension of the Lie algebra at any point of $\mathcal M(c)$ depends only on $s$. According to the Orbit Theorem (see \cite{Agrachev:2004aa}), the dimension of the Lie algebra is constant along any orbit. But according to the particular form of the vector fields $\mathbf Z_c^i$ (whose last $n$ components form a basis of $\mathbf R^n$), the projection of any orbit on $\mathcal S(c)$ turns out to be the whole set $\mathcal S(c)$ (or, in other words, for any $s\in\mathcal S(c)$ and for any orbit, there is a point of the orbit for which the last component is $s$). Assume now that ${\rm dim}\,{\rm Lie}_{\zeta^\ast}\mathcal Z(c)=6+n$ at some particular point $\zeta^\ast:=(R^\ast,\mathbf r^\ast,s^\ast)\in\mathcal M(c)$. Then, according to the Orbit Theorem, for any $s\in\mathcal S(c)$, there exists at least one point $(R_s,\mathbf r_s,s)\in\mathcal M(c)$ such that  ${\rm dim}\,{\rm Lie}_{(R_s,\mathbf r_s,s)}\mathcal Z(c)=6+n$. But since the dimension of the Lie algebra does not depend on the variables $R$ and $\mathbf r$, we conclude that ${\rm dim}\,{\rm Lie}_{\zeta}\mathcal Z(c)=6+n$ for all $\zeta\in\mathcal M(c)$.
\end{proof}

\begin{definition}
\label{def:cont:SC}
We say that $c$, a SC in $C(n)$ (for some integer $n$) is controllable if there exists $\zeta\in\mathcal M(c)$ such that ${\rm dim}\,{\rm Lie}_{\zeta}\mathcal Z(c)=6+n$.
\end{definition}

It is obvious that for a SC to be controllable, the integer $n$ has to be larger or equal to 2. 
The following result is quite classical (a proof can be found in \cite{Chambrion:2010aa}):
\begin{proposition}
\label{prop:prop12}
Let $c\in C(n)$ (for some integer $n$) be controllable (with the usual notation $c:=(\varrho,\vartheta,\mathcal V)$, $\mathcal V:=(\mathbf V_1,\ldots,\mathbf V_n)$ and $\vartheta_s:=\vartheta+\sum_{i=1}^ns_i\mathbf V_i$ for every $s\in\mathcal S(c)$). Then for any given continuous function $t\in[0,T]\mapsto (\bar R(t),\bar{\mathbf r}(t),\bar s(t))\in {\rm SO}(3)\times\mathbf R^3\times \mathcal S(c)$ and for any $\varepsilon>0$, there exist $n$ $C^1$ functions $\lambda_i:[0,T]\to\mathbf R$ ($i=1,\ldots,n$) such that:
\begin{enumerate}
\item $\sup_{t\in[0,T]}\Big(\| \bar R(t)-R(t)\|_{{\rm M}(3)}+\|\bar{\mathbf r}(t)-\mathbf r(t)\|_{\mathbf R^3}+\|\vartheta_{\bar s(t)}-\vartheta_{s(t)}\|_{C^1_0(\mathbf R^3)^3}\Big)<\varepsilon;$
\item $R(T)=\bar R(T)$, $\mathbf r(T)=\bar{\mathbf r}(T)$ and $s(T)=\bar s(T)$; 
\end{enumerate}
where $t\in[0,T]\mapsto(R(t),\mathbf r(t),s(t))\in\mathcal M(c)$ is the unique solution to the ODE \eqref{dynamics:3} with Cauchy data $R(0)=\bar R(0)\in{\rm SO}(3)$, $\mathbf r(0)=\bar{\mathbf r}(0)\in\mathbf R^3$, $s(0)=\bar s(0)\in\mathcal S(c)$.
\end{proposition}

Let us mention some other quite elementary properties that will be used later on:

\begin{proposition} 
\label{properties_of_SC}
\begin{enumerate}
\item If $c:=(\varrho,\vartheta,\mathcal V)\in\mathcal C(n)$ ($n\geq 2)$ is a controllable SC with $\mathcal V:=(\mathbf V_1,\ldots,\mathbf V_n)\in(C^1_0(\mathbf R^3)^3)^n$ then any $c^+:=(\varrho,\vartheta,\mathcal V^+)\in\mathcal C(n+1)$  such that $\mathcal V^+:=(\mathbf V_1,\ldots,\mathbf V_n,\mathbf V_{n+1})\in (C^1_0(\mathbf R^3)^3) ^{n+1}$ (for some $\mathbf V_{n+1}\in C^m_0(\mathbf R^3)^3$) is a controllable SC as well.
\item If $c:=(\varrho,\vartheta,\mathcal V)\in\mathcal C(n)$ ($n\geq 2)$ is a controllable SC, then for any $\vartheta^\ast\in\{\vartheta+\sum_{i=1}^ns_i\mathbf V_i,\,s\in\mathcal S(c)\}$ the element  $c^\ast:=(\varrho,\vartheta^\ast,\mathcal V)$ belongs to  $\mathcal C(n)$ and is a controllable SC as well.
\item If $c:=(\varrho,\vartheta,\mathcal V)\in\mathcal C(n)$ ($n\geq 2)$ is a controllable SC, then all of the controllable SC  $c^\ast:=(\varrho,\vartheta,\mathcal V^\ast)$ in $\mathcal C(n)$ ($\mathcal V^\ast\in(C^1_0(\mathbf R^3)^3)^n$) form an open dense subset of the section $\pi^{-1}(\{(\varrho,\vartheta)\})$ (for the induced topology).
\item If there exists a SC in $\mathcal C(n)$ for some $n\geq 2$ then, for any $k\geq n$, all of the controllable SC in $\mathcal C(k)$ form an open dense subset of $\mathcal C(k)$ (for the induced topology).
\end{enumerate}
\end{proposition}

\begin{proof}
The two first assertions are obvious so let us address directly the third point. Denote $\mathcal E_k$ ($k$ positive integer) the set of all of the vectors fields on $\mathcal M(c)$ obtained as Lie brackets of order lower or equal to $k$ from elements of $\mathcal Z(c)$. Then, consider the determinants of all of the different families of $6+n$ elements of $\mathcal E_k$ as analytic functions in the variable $\mathcal V$ (the other variables $\varrho$, $\vartheta$ and $s=0$ being fixed). Since $c$ is controllable, there exist at least one $k$ and one family of $6+n$ elements in $\mathcal E_k$ whose determinant is nonzero. According to Corollary~\ref{cor:2} and basic properties of analytic functions (see \cite{Whittlesey:1965aa}), the determinant may vanish only in a closed subset with empty interior of the section $\pi^{-1}(\{(\varrho,\vartheta)\})$ (for the induced topology). The proof of the last point is similar.
\end{proof}
\subsection{Building a controllable SC}
In this subsection, we are interested in computing the Lie brackets of first order $[\mathbf Z_c^i(R,s),\,\mathbf Z_c^j(R,s)]$ at $(R,s)=({\rm Id},0)$, for a particular SC. 
\subsubsection*{General computations}
Starting from identity \eqref{identity:lie_brackets} and focusing on the second term in the right hand side, we have, for all $c\in\mathcal C(n)$ ($n\geq 2$) , all $s\in\mathcal S(c)$ and all $i,j=1,\ldots,n$:
\begin{multline}
\label{lie_brackets}
\partial_{s_i}\mathbf X_j(\mathbf c)-\partial_{s_j}\mathbf X_i(\mathbf c)=
\mathbb M^r(\mathbf c)^{-1}\Big[(\partial_{s_j}\mathbb M^r(\mathbf c)\mathbf X_i(\mathbf c)-\partial_{s_i}\mathbb M^r(\mathbf c)\mathbf X_j(\mathbf c))\\
+(\partial_{s_j}\mathbb N(\mathbf c)\mathbf f_i-\partial_{s_i}\mathbb N(\mathbf c)\mathbf f_j)\Big].
\end{multline}From the decomposition $\mathbb M^r(\mathbf c):=\mathbb M^r_b(\mathbf c)+\mathbb M^r_f(\mathbf c)$, we deduce that:
\begin{multline}
\label{decomp}
\partial_{s_j}\mathbb M^r(\mathbf c)\mathbf X_i(\mathbf c)-\partial_{s_i}\mathbb M^r(\mathbf c)\mathbf X_j(\mathbf c)=\\
(\partial_{s_j}\mathbb M^r_b(\mathbf c)\mathbf X_i(\mathbf c)-\partial_{s_i}\mathbb M^r_b(\mathbf c)\mathbf X_j(\mathbf c))+(\partial_{s_j}\mathbb M^r_f(\mathbf c)\mathbf X_i(\mathbf c)-\partial_{s_i}\mathbb M^r_f(\mathbf c)\mathbf X_j(\mathbf c)).
\end{multline}
We can easily compute the first term in the right hand side when $s=0$. Thus, for all $i,j=1,\ldots,n$, we have:
\begin{equation}
\label{rigid_lie_brackets}
[\partial_{s_j}\mathbb M^r_b(\mathbf c)\mathbf X_i(\mathbf c)-\partial_{s_i}\mathbb M^r_b(\mathbf c)\mathbf X_j(\mathbf c)]_{s=0}=\eta_j
\begin{bmatrix}{\rm Id}&0\\
0&0\end{bmatrix}\mathbf X_i(\mathbf c)|_{s=0}-\eta_i
\begin{bmatrix}{\rm Id}&0\\
0&0\end{bmatrix}\mathbf X_j(\mathbf c)|_{s=0},
\end{equation}
where $\eta_i:=2\int_{B}\varrho(x)\, \mathbf V_i(x)\cdot \varTheta(x)\,{\rm d}x$. 

\subsubsection*{Rigid shell's deformation} We consider now shape changes that reduce to rigid displacements on the swimmer's boundary $\Sigma$ (but obviously not inside the body for the self-propelled constraints not to be violated). The idea of using such deformations stemmed from the reading of the article \cite{Kozlov:2003aa}. However, rigid deformations of the shell can not be considered within the modeling described in Section~\ref{section:SC} (because rigid deformations do not fit with the general form of the diffeomorphisms described in Section~\ref{section:SC}). Let us explain how to deal with this difficulty.

Let $c:=(\varrho,\vartheta,\mathcal V)\in\mathcal C(6)$ be given such that $\mathcal V:=(\mathbf V_1,\ldots,\mathbf V_6)$ with $\mathbf V_i|_{\Sigma}(x):=\mathbf e_i\times(x+\vartheta(x))$ ($i=1,2,3$) and $\mathbf V_i|_\Sigma(x):=\mathbf e_{i-3}$ ($i=4,5,6$) (Proposition~\ref{rectif_vector_fields} in the appendix guarantees the existence of such vector fields satisfying furthermore $\eta_i=0$ for all $i=1,\ldots,6$). Thus, the shape changes we are considering read $\varTheta_s:={\rm Id}+\vartheta+\sum_{i=1}^6s_i\mathbf V_i$ ($s\in\mathcal S(c)$). Let us define also $\varXi_s:=R_s({\rm Id}+\vartheta)+\boldsymbol\tau_s$ where $R_s:=\exp({s_1\hat{\mathbf e}_1})\exp(s_2{\hat{\mathbf e}_2})\exp(s_3{\hat{\mathbf e}_3})\in{\rm SO}(3)$ and $\boldsymbol\tau_s:=\sum_{i=1}^3s_{i+3}\mathbf e_i\in\mathbf R^3$.
Thus, $\varTheta_s$ is a diffeomorphism which can be achieved within our modeling while $\varXi_s$ is a {\it true} rigid deformation. 

In order to determine the expressions of the terms $\partial_{s_i}\mathbb M^r_f(\mathbf c)|_{s=0}$ ($i=1,\ldots,6$) arising in the computation of the Lie brackets, we have  to compare (using the notation of Lemma~\ref{lemma:import}) $\partial_{s_i}\varPhi(\mathbf q^1_s)|_{s=0}$ and $\partial_{s_i}\varPhi(\mathbf q^2_s)|_{s=0}$ ($i=1,\ldots,6$) where:
\begin{itemize}
\item $\mathbf q_s^1:=(\xi^1_s,\mathcal W^1_s)$ with $\xi^1_s:=\varXi_s-{\rm Id}$, $\mathcal W^1_s:=(\mathbf W^{1,1}_s,\mathbf W^{1,2}_s)$ and $\mathbf W^{1,1}_s,\,\mathbf W^{1,2}_s\in\{\mathbf e_i\times\varXi_s,\,\mathbf e_i,\,i=1,2,3\}$ (settings corresponding to a true rigid deformation of the shell);
\item $\mathbf q^2_s:=(\xi^2_s,\mathcal W^2_s)$ with 
$\xi^2_s:=\varTheta_s-{\rm Id}=\vartheta+\sum_{i=1}^ns_i\mathbf V_i$, $\mathcal W^2_s:=(\mathbf W^{2,1}_s,\mathbf W^{2,2}_s)$ and $\mathbf W^{2,1}_s,\mathbf W^{2,2}_s\in\{\mathbf e_i\times\varTheta_s,\,\mathbf e_i,\,i=1,2,3\}$ (settings allowed in our modeling).
\end{itemize}
Remark that $\partial_{s_i}\varTheta_s|_{s=0}=\partial_{s_i}\varXi_s|_{s=0}$ on $\Sigma$ for all $i=1,\ldots,6$, so the deformations are tangent at $s=0$. It entails that $\partial_{s_i}\mathbf q^1_s|_{s=0}=\partial_{s_i}\mathbf q^2_s|_{s=0}$. Since, in addition, $\mathbf q^1_{s=0}=\mathbf q^2_{s=0}$ and
$\partial_{s_i}\varPhi(\mathbf q^k_s)|_{s=0}=\langle\partial_{\mathbf q}\varPhi(\mathbf q^k_s),\partial_{s_i}\mathbf q^k_{s}\rangle|_{s=0}$ for $k=1,2$ and $i=1,\ldots,6$, we deduce that $\partial_{s_i}\varPhi(\mathbf q^1_s)|_{s=0}=\partial_{s_i}\varPhi(\mathbf q^2_s)|_{s=0}$ for all $i=1,\ldots,6$. 

We apply the diffeomorphism $\varXi_s:=R_s({\rm Id}+\vartheta)+\boldsymbol\tau_s$ to $\bar B$ to obtain the domain of the deformed swimmer. We denote $\mathcal B^\diamond:=({\rm Id}+\vartheta)(B)$ (which can be seen as the shape of the swimmer at rest, i.e. when $s=0$), $\Sigma^\diamond:=\partial\mathcal B^\diamond$, $\mathcal F^\diamond:=\mathbf R^3\setminus\bar{\mathcal B}^\diamond$ and $\mathcal B_s:=\varXi_s(B)$, $\mathcal F_s:=\mathbf R^3\setminus\bar{\mathcal B}_s$, $\Sigma_s:=\partial\mathcal B_s$.
We seek the potential $\psi^1_{\mathbf c}$, defined and harmonic in $\mathcal F_s$ in the form $\psi^1_{\mathbf c}(x)=\tilde\psi^1_c(R_s^t(x-\boldsymbol\tau_s))$ where the function $\tilde\psi^1_c$ defined on $\mathcal F^\diamond$ has to be determined. 
It is obvious that $\tilde\psi^1_c$ is harmonic in $\mathcal F^\diamond$ and we have only to determine the boundary conditions on $\Sigma^\diamond$. For all $y\in\Sigma^\diamond$, we denote $x:=R_sy+\boldsymbol\tau_s\in\Sigma_s$. The relation $\mathbf n|_{\Sigma_s}(R_sy+\boldsymbol\tau_s)=R_s\mathbf n|_{\Sigma^\diamond}(y)$ entails that $\nabla\tilde\psi^1_c(y)\cdot \mathbf n|_{\Sigma^\diamond}(y)=\nabla\psi^1_{\mathbf c}(R_sy+\boldsymbol\tau_s)\cdot \mathbf n|_{\Sigma_s}(R_sy+\boldsymbol\tau_s)=\nabla\psi^1_{\mathbf c}(x)\cdot \mathbf n(x)$,
and this quantity has to be equal to $(\mathbf e_1\times x)\cdot \mathbf n(x)$. We deduce after some elementary algebra that $
\partial_n\tilde\psi^1_c(y)=(R_s^t\mathbf e_1\times y)\cdot \mathbf n(y)+(R_s^t\mathbf e_1\times R_s^t\boldsymbol\tau_s)\cdot \mathbf n(y)$.
Proceeding similarly and with obvious notation, we obtain more generally that:
\begin{alignat*}{3}
\partial_n\tilde\psi^i_c(y)&=(R_s^t\mathbf e_i\times y)\cdot \mathbf n(y)+(R_s^t\mathbf e_i\times R_s^t\boldsymbol\tau_s)\cdot \mathbf n(y),&\qquad&(i=1,2,3),\\
\text{and}\qquad \partial_n\tilde\psi^i_c(y)&=R_s^t\mathbf e_{i-3}\cdot \mathbf n(y),&&(i=4,5,6).
\end{alignat*}
Denoting $\mathbb M_f^\diamond:=\mathbb M^r_f(\mathbf c)|_{s=0}$, we get the relations:
$$\mathbb M^r_f(c,s)=\begin{bmatrix}{\rm Id}&\hat{\boldsymbol\tau}_s\\0&{\rm Id}\end{bmatrix}\begin{bmatrix}R_s&0\\0&R_s\end{bmatrix}\mathbb M_f^\diamond\begin{bmatrix}R_s^t&0\\0&R_s^t\end{bmatrix}\begin{bmatrix}{\rm Id}&0\\ -\hat {\boldsymbol\tau}_s&{\rm Id}\end{bmatrix},$$
and  then, differentiating with respect to $s_i$:
\begin{subequations}
\label{fluid_lie_brackets}
\begin{alignat}{3}\partial_{s_i}\mathbb M^r_f(c,s)|_{s=0}&=\begin{bmatrix}\hat {\mathbf e}_i&0\\
0&\hat{\mathbf e}_i\end{bmatrix}\mathbb M_f^\diamond-\mathbb M_f^\diamond\begin{bmatrix}\hat {\mathbf e}_i&0\\
0&\hat{\mathbf e}_i\end{bmatrix},&\qquad&(i=1,2,3),\\
\partial_{s_i}\mathbb M^r_f(c,s)|_{s=0}&=\begin{bmatrix}0&\hat{\mathbf e}_{i-3}\\
0&0\end{bmatrix}\mathbb M_f^\diamond-\mathbb M_f^\diamond\begin{bmatrix}0&0\\
\hat{\mathbf e}_{i-3}&0\end{bmatrix},&&(i=4,5,6).
\end{alignat}
\end{subequations}
The reasoning is quite similar for the entries of the matrix $\mathbb N(\mathbf c)$. We only have to be careful that the vector fields  $\mathbf W^{2,2}_s$ can actually not depend on $s$ (once more, this is due to our modeling). The elements of the matrix $\partial_{s_j}\mathbb N(\mathbf c)|_{s=0}\mathbf f_i-\partial_{s_i}\mathbb N(\mathbf c)|_{s=0}\mathbf f_j$ read (still with the notation of  Lemma~\ref{lemma:import}) $\partial_{s_j}\varPhi(\mathbf q^i_s)|_{s=0}-\partial_{s_i}\varPhi(\mathbf q^j_s)|_{s=0}$ ($i,j=1,\ldots,6$) with this time $\mathbf q^i_s:=(\xi_s,\mathcal W^i_s)$, $\xi_s:=\varXi_s-{\rm Id}$, $\mathcal W^i_s:=(\mathbf W^{1}_s,\mathbf W^{i,2})$, $\mathbf W^{1}_s\in\{\mathbf e_k\times \varXi_s,\,\mathbf e_k,\,k=1,2,3\}$ and $\mathbf W^{i,2}:=\mathbf e_i\times ({\rm Id}+\vartheta)$ ($i=1,2,3$) or $\mathbf W^{i,2}=\mathbf e_{i-3}$ ($i=4,5,6$). The only difference with the elements of $\mathbb M^r_f(\mathbf c)$ being that $\mathbf W^{i,2}$ does not depend on $s$ for $i=1,2,3$, we wish to reuse the preceding calculations. Invoking the chain rule, we have to subtract to $\partial_{s_j}\varPhi(\mathbf q^i_s)|_{s=0}-\partial_{s_i}\varPhi(\mathbf q^j_s)|_{s=0}$ the quantity $\langle\partial_{\mathbf W^2}\varPhi(\mathbf q^i_s),\partial_{s_j}\mathbf W^{i,2}_s\rangle|_{s=0}-\langle\partial_{\mathbf W^2}\varPhi(\mathbf q^j_s),\partial_{s_i}\mathbf W^{j,2}_s\rangle|_{s=0}$ which is quite easy to determine because $\varPhi$ is linear in the variable $\mathbf W^2$. Thus, this last expression is merely equal to $\varPhi({\mathbf q}_{i,j})$ ($1\leq i,j\leq 6$) with $\mathbf q_{i,j}:=(\vartheta,\mathcal W_{i,j})$, $\mathcal W_{i,j}:=(\mathbf W^1,\mathbf W^2_{i,j})$, $\mathbf W^1\in\{\mathbf e_k\times ({\rm Id+\vartheta}),\,\mathbf e_k,\,k=1,2,3\}$ and 
$$\mathbf W^2_{i,j}:=
\begin{cases}(\mathbf e_i\times\mathbf e_j)\times({\rm Id}+\vartheta)&\text{if }1\leq i,j\leq 3,\\
(\mathbf e_i\times\mathbf e_j)&\text{if }4\leq i\leq 6\,,1\leq j\leq 3,\\
(\mathbf e_i\times\mathbf e_j)&\text{if }1\leq i\leq 3\,,4\leq j\leq 6,\\
\mathbf 0&\text{if }4\leq i,j\leq 6.
\end{cases}$$
We eventually obtain that:
\begin{equation}
\label{deux_lie_brackets}
\partial_{s_j}\mathbb N(\mathbf c)|_{s=0}\mathbf f_i-\partial_{s_i}\mathbb N(\mathbf c)|_{s=0}\mathbf f_j=\partial_{s_j}\mathbb M^r_f(\mathbf c)|_{s=0}\mathbf f_i-\partial_{s_i}\mathbb M^r_f(\mathbf c)|_{s=0}\mathbf f_j-\mathbb N^\diamond(\mathbf f_i\star\mathbf f_j),
\end{equation}
where $\mathbb N^\diamond:=\mathbb N(\mathbf c)|_{s=0}=\mathbb M^\diamond_f$ and $\mathbf f_i\star\mathbf f_j:=(\mathbf f_i^1\times \mathbf f_j^1,\mathbf f_i^1\times \mathbf f_j^2-\mathbf f_j^1\times \mathbf f_i^2)^t$ with the notation $\mathbf f_i=(\mathbf f_i^1,\mathbf f_i^2)^t\in\mathbf R^3\times\mathbf R^3$. 
\subsubsection*{Specifying the density and shape}The expression of the $6\times 6$ symmetric added mass matrix $\mathbb M^\diamond_f$ depends only on the domain $\mathcal F^\diamond$ or equivalently on $\mathcal B^\diamond$. As stated in Proposition~\ref{matrix_positive_definite} in the appendix, this matrix is positive definite if we choose $\vartheta$ in such a way that $\mathcal B^\diamond$ has no axis of symmetry. It entails that, up to a change of frame $\mathfrak e$ at the initial time, $\mathbb M^\diamond_f$ can be assumed to be diagonal with positive eigenvalues $\mu_j$ ($j=1,\ldots,6$). On the other hand, denoting by ${\rm S}(3)^+$ the set of the $3\times 3$ symmetric matrices that are positive definite, we can quite easily prove that for any $\mathcal B^\diamond$ (which means for any $\vartheta\in D^1_0(\mathbf R^3)$), the mapping $\varrho\in C^0(\bar{\mathcal B}^\diamond)^+\mapsto \int_{\mathcal B^\diamond}\varrho(\|x\|_{\mathbf R^3}{\rm Id}-x\otimes x){\rm d}x\in{\rm S}(3)^+$ is onto. We deduce that for any $(I_1,I_2,I_3)\in\mathbf R^3$ such that $I_i>0$ for $i=1,2,3$, there exists $\varrho\in C^0(\bar B)^+$ such that the inertia tensor $\mathbb I(\mathbf c)|_{s=0}$ is diagonal, equal to ${\rm diag}(I_1,I_2,I_3)$. 
In this case, the matrix $\mathbb M^r_b(\mathbf c)|_{s=0}$ is diagonal as well, equal to ${\rm diag}(I_1,I_2,I_3,m,m,m,)$. We deduce that the vector fields $\mathbf X_i(\mathbf c)|_{s=0}$ read $-\mu_i/(I_i+\mu_i)\mathbf f_i$ if $i=1,2,3$ and  $-\mu_i/(m+\mu_i)\mathbf f_i$ if $i=4,5,6$. Summarizing \eqref{identity:lie_brackets}, \eqref{lie_brackets}, \eqref{decomp}, \eqref{rigid_lie_brackets} (recall that $\eta_i=0$ for all $i=1,\ldots,6$), \eqref{fluid_lie_brackets} and \eqref{deux_lie_brackets}, we obtain the following expressions for the Lie brackets at $(R,s)=({\rm Id},0)$:
\begin{align*}
[\mathbf Z_c^1,\, \mathbf Z_c^2]&=\frac{I_1I_2(-\mu_1-\mu_2+\mu_3)+\mu_1\mu_2(-I_1-I_2+I_3)}{(\mu_1+I_1)(\mu_2+I_2)(\mu_3+I_3)}\begin{pmatrix}
\hat{\mathbf e}_3\\ \mathbf 0\\ \mathbf 0\end{pmatrix},\\
[\mathbf Z_c^1,\, \mathbf Z_c^3]&=\frac{I_1I_3(\mu_1-\mu_2+\mu_3)+\mu_1\mu_3(I_1-I_2+I_3)}{(\mu_1+I_1)(\mu_2+I_2)(\mu_3+I_3)}\begin{pmatrix}
\hat{\mathbf e}_2\\ \mathbf 0\\ \mathbf 0\end{pmatrix},\\
[\mathbf Z_c^2,\, \mathbf Z_c^3]&=\frac{I_2I_3(\mu_1-\mu_2-\mu_3)+\mu_2\mu_3(I_1-I_2-I_3)}{(\mu_1+I_1)(\mu_2+I_2)(\mu_3+I_3)}\begin{pmatrix}
\hat{\mathbf e}_1\\ \mathbf 0\\ \mathbf 0\end{pmatrix},\\
[\mathbf Z_c^2,\, \mathbf Z_c^4]&=\frac{I_2m(\mu_4-\mu_6)}{(\mu_2+I_2)(\mu_4+m)(\mu_6+m)}\begin{pmatrix}
\mathbf 0\\ {\mathbf e}_3\\ \mathbf 0\\\end{pmatrix},\\
[\mathbf Z_c^3,\, \mathbf Z_c^4]&=\frac{I_3m(\mu_5-\mu_4)}{(\mu_3+I_3)(\mu_4+m)(\mu_5+m)}\begin{pmatrix}
\mathbf 0\\ {\mathbf e}_2\\ \mathbf 0\\\end{pmatrix},\\
[\mathbf Z_c^3,\, \mathbf Z_c^5]&=\frac{I_3m(\mu_5-\mu_4)}{(\mu_3+I_3)(\mu_4+m)(\mu_5+m)}\begin{pmatrix}
\mathbf 0\\ {\mathbf e}_1\\ \mathbf 0\\\end{pmatrix}.
\end{align*}
Since, on the other hand, when $(R,s)=({\rm Id},0)$ we also have $$\mathbf Z^i_c=\begin{pmatrix}-{\mu_i}\hat{\mathbf e}_i/({\mu_i+I_i})\\
\mathbf 0\\ \mathbf f_i\end{pmatrix}\quad\text{ if }\,i=1,2,3,\quad\text{and}\quad\mathbf Z^i_c=\begin{pmatrix}\mathbf 0\\
-{\mu_i}{\mathbf e}_i/({\mu_i+m})\\
\mathbf f_i\end{pmatrix}\quad\text{ if }i=4,5,6,$$
we deduce that a sufficient condition ensuring that dim$\,{\rm Lie}_{({\rm Id},\mathbf 0,0)}\mathcal Z(c)=12$  is that 
\begin{multline}
\label{cond0}
\big[I_1I_2(-\mu_1-\mu_2+\mu_3)+\mu_1\mu_2(-I_1-I_2+I_3)\big]\big[I_1I_3(\mu_1-\mu_2+\mu_3)+\mu_1\mu_3(I_1-I_2+I_3)\big]\\
\big[I_2I_3(\mu_1-\mu_2-\mu_3)+\mu_2\mu_3(I_1-I_2-I_3)\big]\big[I_2m(\mu_4-\mu_6)I^2_3m^2(\mu_5-\mu_4)^2\big]\neq 0.
\end{multline}
According to \cite[pages 152-155]{Lamb:1993aa}, if we specialize now $\mathcal B^\diamond$ to be an ellipsoid, the length of the axes can be chosen in such a way that (i) it has no axis of symmetry (and hence $\mu_i>0$  for $i=1,\ldots,6$), (ii) $\mu_4\neq \mu_5$ and (iii) $\mu_4\neq \mu_6$.  Since $I_i>0$ ($i=1,2,3$) and obviously $m>0$, the condition \eqref{cond0} reduces to:
\begin{multline}
\label{cond}
\big[I_1I_2(-\mu_1-\mu_2+\mu_3)+\mu_1\mu_2(-I_1-I_2+I_3)\big]\big[I_1I_3(\mu_1-\mu_2+\mu_3)+\mu_1\mu_3(I_1-I_2+I_3)\big]\\
\big[I_2I_3(\mu_1-\mu_2-\mu_3)+\mu_2\mu_3(I_1-I_2-I_3)\big]\neq 0.
\end{multline}
As already mentioned, it is always possible to achieve any triplet of positive numbers $(I_1,I_2,I_3)$ with a suitable choice of density, so whatever the values of $\mu_i$ ($i=1,2,3$) are, it is always possible to find $\varrho\in C^0(\bar B)^+$ such that 
 \eqref{cond} holds. It entails, according to the forth point of Proposition~\ref{properties_of_SC}:
\begin{proposition}
\label{everything_is_controllable}
For any integer $n\geq 5$, the set of all the controllable SC is an open dense subset in $\mathcal C(n)$. 
\end{proposition}

Notice that $n=5$ in this Proposition (instead of $n=6$), because we did not use the vector field $\mathbf Z_c^6$ in the computation of the Lie brackets.
\section{Proofs of the Main Results}
\label{sec:main_results}
\subsubsection*{Proof of Proposition~\ref{existence}}
Let a control function $\vartheta$ be given in $AC([0,T],D^1_0(\mathbf R^3))$ and denote $\varTheta:={\rm Id}+\vartheta$. With the notation of Lemma~\ref{lemma:import}, at any time $t$ the entries of the matrix $\mathbb M^r_f(t)$ read $\varPhi(\mathbf q)$ with $\mathbf q:=(\vartheta,\mathcal W)$ and $\mathcal W:=(\mathbf W^1,\mathbf W^2)$, $\mathbf W^i\in\{\mathbf e_i\times\varTheta_t,\,\mathbf e_i,\,i=1,2,3\}$. We deduce that $t\in[0,T]\mapsto\mathbb M^r_f(t)\in{\rm M}(3)$ is in $AC([0,T],{\rm M}(3))$. To get the expression of the elements of the vector $\mathbf N(t)$ we only have to modify $\mathbf W^2$ which has to be equal to $\partial_t\vartheta_t$. It entails that $t\in[0,T]\mapsto\mathbf N(t)\in\mathbf R^6$ is in $L^1([0,T])^6$. It is quite easy to verify that the inertia tensor $\mathbb I(t)$ is in $AC([0,T],{\rm M}(3))$.
Existence and uniqueness of solutions is now straightforward because $t\in[0,T]\mapsto\mathbb M^r(t)^{-1}\mathbf N(t)\in\mathbf R^6$ is in $AC([0,T],\mathbf R^6)$. 

Let us address the stability result. With the same notation as in the statement of Proposition~\ref{existence}, denote $({\boldsymbol\Omega}^j,\mathbf v^j)^t$ the left hand side of identity \eqref{dynamics} with control $\vartheta^j$ and $(\bar{\boldsymbol\Omega},\bar{\mathbf v})^t$ when the control is $\bar\vartheta$.  As $j\to+\infty$, it is clear that $({\boldsymbol\Omega}^j,\mathbf v^j)^t\to(\bar{\boldsymbol\Omega},\bar{\mathbf v})^t$ in $L^1([0,T],\mathbf R^6)$. Then, integrating \eqref{complement} between 0 and $t$ for any $0\leq t\leq T$, we get the estimate $\|\bar R(t)-R^j(t)\|_{{\rm M}(3)}\leq \int_0^t\|\bar R(s)-R^j(s)\|_{{\rm M}(3)}\|\bar{\boldsymbol\Omega}(s)\|_{\mathbf R^3}+\|\boldsymbol\Omega^j(s)-\bar{\boldsymbol\Omega}(s)\|_{\mathbf R^3}{\rm d}s.$ Applying Gr\"onwall inequality, we conclude that $R^j\to\bar R$ in $C([0,T],{\rm M}(3))$ as $j\to+\infty$ and we use again the ODE to prove that $\dot R^j\to\dot{\bar R}$ in $L ^1([0,T])$. Next, it is then easy to obtain the convergence of $\mathbf r^j$ to $\bar{\mathbf r}$ and to conclude the proof.

\subsubsection*{Proof of Theorems~\ref{main_theorem_cont} and \ref{main_theorem_free}} We shall focus on the proof of Theorem~\ref{main_theorem_cont} because it will contain the proof of Theorem~\ref{main_theorem_free}. For any integer $n$, we shall use the notation $\|c\|_{\mathcal C(n)}:=\|\varrho\|_{C^0(\bar B)}+\|\vartheta\|_{C^1_0(\mathbf R^3)^3}+\sum_{i=1}^n\|\mathbf V_i\|_{C^1_0(\mathbf R^3)^3}$
for all $c\in C^0(\bar B)\times C^1_0(\mathbf R^3)^3\times(C^1_0(\mathbf R^3)^3)^n$ with, as usual, $c:=(\varrho,\vartheta,\mathcal V)$ and $\mathcal V:=(\mathbf V_1,\ldots,\mathbf V_n)$.

Let $\varepsilon>0$ and the functions $\bar\varrho\in C^0(\bar B)$, $t\in[0,T]\mapsto\bar\vartheta_t\in D^1_0(\mathbf R^3)$ and $t\in[0,T]\mapsto (\bar R(t),\bar{\mathbf r}(t))\in{\rm SO}(3)\times\mathbf R^3$ be given as in the statement of the theorem. 

Set now $\bar\varrho^1:=\bar\varrho$, $\bar\vartheta^1:=\bar\vartheta_{t=0}$ and $\bar{\mathbf V}^1_1:=\partial_t\bar\vartheta_{t=0}\in C^1_0(\mathbf R^3)^3$. According to the self-propelled constraints \eqref{self-propelled-cond}, it is always possible to find four elements $\bar{\mathbf V}^1_j$ ($j=2,\ldots,5$) in $C^1_0(\mathbf R^3)^3$ such that the element $\bar c^1:=(\bar\varrho^1,\bar\vartheta^1,\bar{\mathcal V}^1)$ be a SC which belongs to ${\mathcal C}(5)$, with $\bar{\mathcal V}^1:=(\bar{\mathbf V}^1_1,\ldots,\bar{\mathbf V}^1_5)$. Then, Proposition~\ref{everything_is_controllable} guarantees that for any $\delta>0$ it is possible to find a SC {\it controllable} in $\mathcal C(5)$, denoted by $c^1:=(\varrho^1,\vartheta^1,\mathcal V^1)$ where $\mathcal V^1:=(\mathbf V^1_1,\ldots,\mathbf V^1_5)$, such that $\|c^1-\bar c^1\|_{\mathcal C(5)}<\delta/2$. 
Since $t\mapsto\partial_t\bar\vartheta_t$ is continuous on the compact set $[0,T]$, it is uniformly continuous. For any $\nu>0$, there exists $\delta_\nu>0$ such that 
$ \|\bar\partial_t\vartheta_t-\bar\partial_t\vartheta_{t'}\|_{C^1_0(\mathbf R^3)^3}< \nu$ providing that $|t-t'|\leq \delta_\nu$.
We then divide the time interval $[0,T]$ into $0=t_1<t_2<\ldots<t_k=T$ such that $|t_{j+1}-t_j|<\delta_\nu$ for $j=1,\ldots,k-1$.
For any $t_1\leq t\leq t_2$, we have the estimate:
\begin{multline*}
\|\bar\vartheta_t-(\vartheta^1+(t-t_1)\mathbf V_1^1)\|_{C^1_0(\mathbf R^3)^3}\leq\|\bar\vartheta_t-(\bar\vartheta^1+(t-t_1)\bar{\mathbf V}^1_1)\|_{C^1_0(\mathbf R^3)^3}\\
+\|\bar\vartheta^1-\vartheta^1\|_{C^1_0(\mathbf R^3)^3}+(t-t_1)\|\bar{\mathbf V}^1_1-\mathbf V_1^1\|_{C^1_0(\mathbf R^3)^3}.
\end{multline*}
On the one hande, we have, for all $t\in [t_1,t_2]$, $\|\bar\vartheta_t-(\bar\vartheta^1+(t-t_1)\bar{\mathbf V}^1_1)\|_{C^1_0(\mathbf R^3)^3}< \nu |t-t_1|$. 
On the other hand, still for $t_1\leq t\leq t_2$ and if we assume that $\delta_\nu<1$, we get: $\|\bar\vartheta^1-\vartheta^1\|_{C^1_0(\mathbf R^3)^3}+(t-t_1)\|\bar{\mathbf V}^1_1-\mathbf V_1^1\|_{C^1_0(\mathbf R^3)^3}\leq \delta/2.$
We introduce now $\bar\varrho^2:=\bar\varrho^1$ and $\bar\vartheta^2:=\bar\vartheta_{t=t_2}$. It is always possible to supplement $\bar{\mathbf V}^2_1:=\partial_t\bar\vartheta_{t_2}$ with vector fields $\bar{\mathbf V}_j^2$ ($j=2,\ldots,5$) in such a way that $\bar c^2:=(\bar\varrho^2,\bar\vartheta^2,\bar{\mathcal V}^2)$ be in ${\mathcal C}(5)$ with the obvious notation $\bar{\mathcal V}^2:=(\bar{\mathbf V}^2_1,\ldots,\bar{\mathbf V}_5^2)$. 

We define $\varrho^2:=\varrho^1$ and $\vartheta^2:=\vartheta^1+(t_2-t_1)\mathbf V^1_1$. For any $t_1\leq t\leq t_2$, Proposition~\ref{properties_of_SC} guarantees that the SC $c^1_t:=(\varrho^1,\vartheta^1+(t-t_1)\mathbf V^1_1,\mathcal V^1)$ is controllable. In particular, for $t=t_2$, there exists an integer $k$ and a family of 11 vector fields in $\mathcal E_k$ (the set of all the Lie brackets of order lower or equal to $k$) such that the determinant of the family is nonzero. But this determinant can be thought of as an analytic function in $\mathcal V^1$. The set $\pi^{-1}(\{(\varrho^1,\vartheta^1)\})$ being an analytic connected submanifold of $(C^1_0(\mathbf R^3)^3)^5$ (see Corollary~\ref{cor:2}), the determinant is nonzero everywhere on this set but maybe in a closed subset of empty interior (for the induced topology). Therefore, it is possible to find $\mathcal V^2\in (C^1_0(\mathbf R^3)^3)^5$ such that $\|\bar c^2-c^2\|_{\mathcal C(5)}<(\delta/2+\nu (t_2-t_1))+\delta/4$, and $c^2:=(\varrho^2,\vartheta^2,\mathcal V^2)$ is controllable. 

 By induction, we can build $\bar c^j$ and $c^j$ ($j=1,2,\ldots,k$) such that  $\|\bar c^j-c^j\|\leq \delta/2+\sum_{i=2}^k{\delta}/{2^i}+\nu (t_{i}-t_{i-1})<\delta+\nu T$. We choose $\delta$ and $\nu$ in such a way that $\delta+\nu T<\varepsilon/2$ and we define $t:[0,T]\mapsto \tilde\vartheta_t\in D^1_0(\mathbf R^3)$ and $t\in[0,T]\mapsto \tilde c_t\in\mathcal C(5)$ as continuous, piecewise affine functions by $\tilde\vartheta_t:=\vartheta^j+(t-t_j)\mathbf V^j_1$ and $\tilde c_t:=(\tilde\varrho,\tilde\vartheta_t,\tilde{\mathcal V}_t)$ with $\tilde\varrho:=\varrho^j=\varrho^1$, $\tilde{\mathcal V}_t:=\mathcal V^j$ if $t\in [t_j,t_{j+1}]$ ($j=1,\ldots,k-1$).  Notice that for any $t\in[0,T]$, (i) $\|\bar\vartheta_t-\tilde\vartheta_t\|_{C^1_0(\mathbf R^3)^3}<\varepsilon/2$ and (ii) $\tilde c_t$ is controllable.
 
Definition~\ref{def:cont:SC} 	and Proposition~\ref{prop:prop12} ensure that, on every   interval $[t_j,t_{j+1}]$ ($j=1,\ldots,k-1$), there exist five $C^1$ functions $\lambda^j_i:[t_j,t_{j+1}]\mapsto \mathbf R$ ($i=1,\ldots,5$) such that the solution $(R_j,\mathbf r_j,s^j):[t_j,t_{j+1}]\to {\rm SO}(3)\times\mathbf R^3\times\mathbf R^5$ to the ODE \eqref{dynamics:3}  with vector fields $\mathbf Z_{c^j}^i(R_j,s)$ and Cauchy data $R_j(t_j)=\bar R(t_j)$, $\mathbf r_j(t_j)=\bar{\mathbf r}(t_j)$ and $s^j(t_j)=0$ satisfy:
\begin{enumerate}
\item 
$\sup_{t\in[t_j,t_{j+1}]}\Big(\| \bar R(t)-R_j(t)\|_{{\rm M}(3)}+\|\bar{\mathbf r}(t)-\mathbf r_j(t)\|_{\mathbf R^3}+\|\tilde\vartheta_t-\vartheta_t^j\|_{C^1_0(\mathbf R^3)^3}\Big)<\varepsilon/2$ with $\vartheta^j_t:=\vartheta^j+\sum_{i=1}^5s^j_i(t)\mathbf V^j_i$;
\item $R_j(t_{j+1})=\bar R(t_{j+1})$, $\mathbf r_j(t_{j+1})=\bar{\mathbf r}(t_{j+1})$ and $s^j(t_{j+1})=(t_{j+1}-t_j,0,0,0,0)^t
$;
\end{enumerate}
 With these settings, the functions $t\in[0,T]\mapsto\breve\vartheta_t\in D^1_0(\mathbf R^3)$, $\breve R:[0,T]\to{\rm SO}(3)$ and $\breve{\mathbf r}:[0,T]\to\mathbf R^3$ defined by $\breve\vartheta_t:=\vartheta^j_t$, $\breve R(t):=R_j(t)$ and $\breve{\mathbf r}(t):=\mathbf r_j(t)$ if $t\in[t_j,t_{j+1}]$ ($j=1,\ldots,k-1$) are continuous, piecewise $C^1$. We extend also every functions $s^j$ on $[0,T]$ by setting $s^j(t):=0$ if $t\in[0,t_j[$ and $s^j(t):=(t_{j+1}-t_j,0,0,0,0)^t$ if 
 $t\in]t_{j+1},T]$. They are continuous, piecewise $C^1$ as well.
 
It remains to smooth the function $\breve \vartheta_t$ (and hence also $\breve R$ and $\breve{\mathbf r}$). We can extend the functions $\lambda_i^j$ on the whole interval $[0,T]$ by merely setting $\lambda_i^j(t)=0$ if $t\notin[t_j,t_{j+1}]$. Then, denoting $(\mathbf F^j_i)_{1\leq i\leq 5\atop 1\leq j\leq k}$ the canonical basis of $(\mathbf R^5)^k$ and $\breve S:=\sum_{j=1}^k\sum_{i=1}^5s_i^j\mathbf F_i^j\in (\mathbf R^5)^k$, we get that $(\breve R,\breve{\mathbf r},\breve S)$ is a Carath\'eodory's solution to the following equation on $[0,T]$:
\begin{equation}
\label{la_haut}
\frac{d}{dt}\begin{pmatrix}\breve R\\ \breve{\mathbf r}\\ \breve S\end{pmatrix}=
\sum_{j=1}^k\sum_{i=1}^5{\lambda_i^j}(t)\mathbf T^j_i(\breve R,\breve S),
\end{equation}
where $\mathbf T^j_i(\breve R,\breve S):=(\breve R\hat{\mathbf X}^1_i(c^j,s^j),\breve R{\mathbf X}^2_i(c^j,s^j),\mathbf F_i^j)\in T_R{\rm SO}(3)\times\mathbf R^3\times (\mathbf R^5)^k$. 
Let $\check\lambda_i^j$ denote analytic approximations of the functions $\lambda_i^j$ in $L^1([0,T])$ and denote $(\check R,\check{\mathbf r},\check{S})$ the corresponding analytic solution to system \eqref{la_haut} with $\check S:=(\check s_1^1,\ldots,\check s_5^1,\check s_1^2,\ldots,\check s_5^2,\ldots,\check s_1^k,\ldots,\check s_5^k)$ and $\check\vartheta_t:=\vartheta_0^1+\sum_{j=1}^k\sum_{i=1}^5 \check s^j_i(t)\mathbf V_i^j$ which is analytic from $[0,T]$ to $C^1_0(\mathbf R^3)^3$. According to Lemma~\ref{LEM:converg}, $\breve\vartheta_t-\check\vartheta_t$ can be made arbitrarily small in $AC([0,T],C^1_0(\mathbf R^3)^3)$, providing that  the functions $\check\lambda_i^j$ are close enough to $\lambda_i^j$ in $L^1([0,T])$. Notice that $\check\vartheta$ does probably not satisfy the self-propelled constraints \eqref{self-propelled-cond} (especially at the times $t_j$, $j=2,\ldots,k-1$). It remains to invoke Proposition~\ref{prop:make_allowable} and the continuity of the input-output function (second point of Proposition~\ref{existence}) to conclude that there exists $t\in[0,T]\mapsto\vartheta_t\in D^1_0(\mathbf R^3)$ analytic, satisfying \eqref{self-propelled-cond} and such that 
$$\sup_{t\in[0, T]}\Big(\| R(t)-\breve R(t)\|_{{\rm M}(3)}+\|{\mathbf r}(t)-\breve{\mathbf r}(t)\|_{\mathbf R^3}+\|\vartheta_t-\breve\vartheta_t\|_{C^1_0(\mathbf R^3)^3}\Big)<\varepsilon/2,$$
where $(R,\mathbf r):[0,T]\mapsto {\rm SO}(3)\times\mathbf R^3$ is the solution to System~\eqref{main_dynamics} with initial data $(R(0),\mathbf r(0))=(\bar R(0),\bar{\mathbf r}(0))$ and control $\vartheta$. The proof is then complete. 
\section{Conclusion}
\label{sec:conclusion}
In this paper, we have proved that,  for a 3D shape changing body, the ability of swimming (i.e. not only moving but  tracking any given trajectory) in a vortex free environment is generic. The genericity refers to the shape of the body, its density and the basic movements (at most five) required for swimming. This result is part of a series of articles \cite{Munnier:2008aa, Munnier:2008ab, Munnier:2010aa, Chambrion:2010aa} studying locomotion in a potential flow and the next step of this study will be to investigate whether this controllability result can be extended to a flow with vortices.

\appendix

\section{Function spaces}
\label{SEC:diffeo}
\subsubsection*{Classical function spaces} 
\begin{itemize}
\item For any compact set $K\subset\mathbf R^n$ ($n$ a positive integer), the space $C^0(K)$ consists of the continuous functions in $K$ endowed with the norm $\|u\|_{C^0(K)}:=\sup_{x\in K}|u(x)|$. 
The open subset of the positive functions of $C^0(K)$ is denoted $C^0(K)^+$.
\item  For any open set $\Omega\subset\mathbf R^3$ (included $\Omega=\mathbf R^3$), $\mathcal D(\Omega)$ is the space of the smooth ($C^\infty$) functions, compactly supported in $\Omega$. 
\item For any open set $\Omega\subset\mathbf R^3$ (included $\Omega=\mathbf R^3$), the set $C^1_0(\Omega)$ 
is the completion of $\mathcal D(\Omega)$ for the norm $\|u\|_{C^1_0(\Omega)}:=\sup_{x\in\Omega}|u(x)|+\|\nabla u(x)\|_{\mathbf R^3}$. When $\Omega=\mathbf R^3$, we get $C^1_0(\mathbf R^3):=\{u\in C^1(\mathbf R^3)\,:\:|u(x)|\to 0\text{ and }\|\nabla u(x)\|_{\mathbf R^3}\to 0\text{ as }\|x\|_{\mathbf R^3}\to+\infty\}$.
\item The space $C^1_0(\mathbf R^3)^3$ is the Banach space of all of the vector fields in $\mathbf R^3$ whose every component belongs to $C^1_0(\mathbf R^3)$.
\item Let $E$ be an open subset or an embedded submanifold of a Banach space and $T>0$, then $AC([0,T],E)$ consists in the absolutely continuous functions from $[0,T]$ into $E$. It is endowed with the norm 
$\|u\|_{AC([0,T],E)}:=\sup_{t\in[0,T]}\|u_t\|_E+\int_0^T\|\partial_t u_t\|_E{\rm d}t$.
\item $C_0^m(\Omega,{\rm M}(k))$ ($m$ an integer) is the Banach space of the  functions of class $C^m$ in $\mathbf R^3$ valued in ${\rm M}(k)$ (${\rm M}(k)$ stands for the Banach space of the $k\times k$ matrices, $k$ a positive integer) and compactly supported in $\Omega$. 
\item $E^m_0(\Omega,{\rm M}(k))$ stands for the connected component containing the zero function of the open subset $\{M\in C^m_0(\Omega,{\rm M}(k))\,:\det({\rm Id}+M(x))\neq 0\,\,\forall\,x\in\mathbf R^3\}$. 
\end{itemize}
\begin{lemma}
The set $\tilde D^1_0(\mathbf R^3):=\{\vartheta\in C^1_0(\mathbf R^3)^3\;{\rm s.t.}\;{\rm Id}+\vartheta{\rm\; is\; a }\;C^1\;{\rm diffeomorphism\; of\; }\mathbf R^3\}$
is open in $C^1_0(\mathbf R^3)^3$.
\end{lemma}
\begin{proof}
The mapping $\vartheta\in C^1_0(\mathbf R^3)^3\mapsto\delta_{\vartheta}:=\inf_{\mathbf e\in S^2\atop x\in\mathbf R^3}\langle{\rm Id}+ \nabla \vartheta(x),\mathbf e\rangle\cdot \mathbf e\in\mathbf R$ ($S^2$ stands for the unit 2 dimensional sphere)
is well defined and continuous. For any $\vartheta_0\in \tilde D^1_0(\mathbf R^3)$, we have $\delta_{\vartheta_0}>0$ and for all $x,y\in\mathbf R^3$ and $\mathbf e:=(y-x)/|y-x|$ the following estimate holds:
$(y+\vartheta(y)-x-\vartheta(x))\cdot \mathbf e=|y-x|\int_0^1\langle {\rm Id}+\nabla\vartheta(x+t\mathbf e),\mathbf e\rangle\cdot \mathbf e\,{\rm d}t>|y-x|\delta_\vartheta$.
We deduce that ${\rm Id}+\vartheta$ is one-to-one if $\vartheta$ is close enough to $\vartheta_0$. Further, still for $\vartheta$ close enough to $\vartheta_0$, ${\rm Id}+\vartheta$ is a local diffeomorphism (according to the local inversion Theorem) and hence it is onto.
\end{proof}

\begin{definition}
\label{def_D10}
We denote $D^1_0(\mathbf R^3)$ the connected component of $\tilde D^1_0(\mathbf R^3)$ that contains the identically zero function.
\end{definition}

 If $\vartheta\in C^1_0(\mathbf R^3)^3$ is such that $\|\vartheta\|_{C^1_0(\mathbf R^3)^3}<1$, the local inversion Theorem and a fixed point argument ensure that $\rm{Id}+\vartheta$ is a $C^1$ diffeomorphism so we deduce that $D^1_0(\mathbf R^3)$ contains the unit ball of $C^1_0(\mathbf R^3)^3$.

\subsubsection*{Sobolev spaces} For any open exterior domain $\mathcal F$, the weighted Sobolev space $W^1(\mathcal F)$ is defined by $W^1(\mathcal F):=\{u\in \mathcal D'(\mathcal F)\,:\,u/\sqrt{1+|x|^2}\in L^2(\mathcal F),\,\partial_{x_i}u\in L^2(\mathcal F),\,i=1,2,3\}$ (see \cite{Amrouche:1997aa} for details). 

\section{Making shape changes allowable}
\begin{proposition}
\label{prop:make_allowable}
Let $t\in[0,T]\mapsto \vartheta^j_t\in D^1_0(\mathbf R^3)$ for $j=1,\ldots,+\infty$, be a sequence of absolutely continuous functions (respectively of class $C^p$ for $p=1,\ldots,+\infty$ or analytic) which converges to  $t\in[0,T]\mapsto \vartheta^\dag_t\in D^1_0(\mathbf R^3)$ in $AC([0,T], D^1_0(\mathbf R^3))$. Assume that for some function $\varrho\in C^0(\bar B)^+$, the pair $(\varrho,\vartheta^\dag)$ satisfies \eqref{self-propelled-cond}. Then, there exists a sequence $t\in[0,T]\mapsto \bar\vartheta^j_t\in D^1_0(\mathbf R^3)$ ($j=1,\ldots,+\infty$) in $AC([0,T],D^1_0(\mathbf R^3))$ (respectively of class $C^p$ for $p=1,\ldots,+\infty$ or analytic) such that (i) for any $j=1,\ldots,+\infty$, the pair $(\varrho,\bar\vartheta^j)$ satisfies \eqref{self-propelled-cond} and (ii) $\bar\vartheta^j\to \vartheta^\dag$ in $AC([0,T], D^1_0(\mathbf R^3))$.
\end{proposition}

\begin{proof}Denote $m:=\int_{B}\varrho\,{\rm d}x$ and, for every $j=1,\ldots,+\infty$, $\varTheta^j_t={\rm Id}+\vartheta^j_t$ and $\mathbf r^j(t):=(1/m)\int_{B}\varrho\,\varTheta^j_t\,{\rm d}x$. Then, for any continuous function $t\in[0,T]\mapsto R(t)\in{\rm SO}(3)$, we have $\int_{B}\varrho\,R(t)\tilde\varTheta^j_t\,{\rm d}x=\mathbf 0$ where $\tilde\varTheta^j_t:=\varTheta^j-\mathbf r^j(t)$. Let us now determine an absolutely continuous function $t\in[0,T]\mapsto R^j(t)\in{\rm SO}(3)$ such that  we have also $\int_{B}\varrho[\partial_t(R^j(t)\tilde\varTheta^j_t)\times R^j(t)\tilde\varTheta^j_t]{\rm d}x=\mathbf 0$ for all $t\in]0,T[$.
Introducing $\mathbb I(\tilde\varTheta^j_t):=\int_{B}\varrho[|\tilde\varTheta^j_t|^2{\rm Id}-\tilde\varTheta^j_t\otimes\tilde\varTheta^j_t]{\rm d}x$ (an inertia tensor with $\varrho>0$, so always invertible) and $\mathbf G(\tilde\varTheta^j_t,\partial_t\tilde\varTheta^j_t):=\mathbb I(\tilde\varTheta^j_t)^{-1}\int_{B}\varrho\,\partial_t\tilde\varTheta^j_t\times\tilde\varTheta^j_t{\rm d}x$, the identity above can be turned into $\dot R^j(t)=R^j(t)\hat{\mathbf G}(\tilde\varTheta^j_t,\partial_t\tilde\varTheta^j_t)$ ($0<t<T$).
Since ${\rm SO}(3)$ is compact, this ODE supplemented with the Cauchy data $R(0)={\rm Id}$ admits a unique solution defined for all $t\in[0,T]$. The solution has the same regularity as $t\in[0,T]\mapsto \vartheta^j_t\in D^1_0(\mathbf R^3)$. Besides, basic estimates allow to prove that $(R^j,\mathbf r^j)\to ({\rm Id},\mathbf 0)$ in $AC([0,T],{\rm SO}(3)\times\mathbf R^3)$ and next that $\tilde\vartheta^j\to \vartheta^\dag$ in $AC([0,T],C^1_0(\mathbf R^3)^3)$. Notice that, at this point, we probably have $\tilde\vartheta^j\notin D^1_0(\mathbf R^3)$. Let now $\Omega$ be a large ball containing $\cup_{t\in[0,T]\atop j\in\mathbf N}\tilde\varTheta^j_t(\bar B)$ and $\Omega'$ be an even larger ball containing $\Omega$. Consider $\chi$ a cut-off function such that $0\leq \chi\leq 1$, $\chi = 1$ in $\Omega$ and $\chi=0$ in $\mathbf R^3\setminus \bar\Omega'$. Define $\bar\varTheta^j$ as the flow associated with the Cauchy problem $\dot X(t,x)=\chi(x)\partial_t\tilde\vartheta^j_t+(1-\chi(x))\partial_t\vartheta^\dag$, $X(0,x)=\varTheta^\dag_{t=0}(x)$, ($x\in\mathbf R^3$) and $\bar\vartheta^j:=\bar\varTheta^j-{\rm Id}$. Since $\bar\vartheta^j_{t=0}=\vartheta^\dag_{t=0}$, we deduce that $\bar\vartheta^j_t\in D^1_0(\mathbf R^3)$ for all $t\in[0,T]$ and the sequence $t\in[0,T]\mapsto \bar\vartheta^j_t\in D^1_0(\mathbf R^3)$ complies with the requirements of the Proposition. 
\end{proof}
\section{Making vector fields allowable}
Let a triplet $(\varrho,\vartheta,\mathcal V)\in C^0(\bar B)^+\times D^1_0(\mathbf R^3)\times(C^m_0(\mathbf R^3)^3)^n$ be given such that $\int_B\varrho\,\varTheta\,{\rm d}x=\mathbf 0$ where $\varTheta={\rm Id}+\vartheta$. Recall that $\Sigma=\partial B$. 
\begin{proposition}
\label{rectif_vector_fields}
It is always possible to define new vector fields $\mathbf V_i^\ast$ in such a way that  (i) $\mathbf V_i^\ast|_\Sigma=\mathbf V_i|_{\Sigma}$, (ii) $(\varrho,\vartheta,\mathcal V^\ast)\in\mathcal C(n)$ with $\mathcal V^\ast:=(\mathbf V_1^\ast,\ldots,\mathbf V_n^\ast)$ and (iii) $\int_B\varrho\,\varTheta\cdot\mathbf V_i^\ast\,{\rm d}x=0$.
\end{proposition}
\begin{proof}
Arguing like in the proof of Theorem~\ref{theorem:1}, we can easily show that the mapping $\mathbf W\in C^1_0(B)^3\mapsto(\int_{B}\varrho\, \mathbf W\,{\rm d}x,\int_{B}\varrho\, \varTheta\times\mathbf W\,{\rm d}x,\int_B\varrho\,\varTheta\cdot\mathbf W\,{\rm d}x)\in\mathbf R^3\times\mathbf R^3\times\mathbf R$ is onto with infinite dimensional kernel. Hence, it is always possible to find a vector field $\mathbf W_1\in C^1_0(B)^3$ satisfying $\int_{B}\varrho\, \mathbf W_1\,{\rm d}x=-\int_B\varrho\, \mathbf V_1\,{\rm d}x$, $\int_{B}\varrho\, \varTheta\times\mathbf W_1\,{\rm d}x=-\int_B\varrho\, \varTheta\times\mathbf V_1\,{\rm d}x$, $\int_B\varrho\, \varTheta\cdot\mathbf W^1\,{\rm d}x=-\int_B\varrho\, \varTheta\cdot\mathbf V^1\,{\rm d}x$
and such that $\{\varTheta|_B\cdot\mathbf e_k,\,\mathbf (\mathbf V_1+\mathbf W_1)|_{B}\cdot\mathbf e_k,\,k=1,2,3\}$ is a free family in $C^1_0(\mathbf R^3)^3$.
We denote $\mathbf V_1^\ast:=\mathbf V_1+\mathbf W_1$.  We can continue with the same idea: The mapping $\mathbf W\in C^1_0(\mathbf R^3)^3\mapsto(\int_{B}\varrho\, \mathbf W\,{\rm d}x,\int_{B}\varrho\, \varTheta\times\mathbf W\,{\rm d}x,\int_{B}\varrho\, \mathbf V_1^\ast\times\mathbf W\,{\rm d}x)\in\mathbf R^3\times\mathbf R^3\times\mathbf R^3$ is onto, also with infinite dimensional kernel. Again, it is possible to find $\mathbf W_2\in C^1_0(B)^3$ satisfying
$\int_{B}\varrho\, \mathbf W_2\,{\rm d}x=-\int_B\varrho\, \mathbf V_2\,{\rm d}x$, $\int_{B}\varrho\, \varTheta\times\mathbf W_2\,{\rm d}x=-\int_B\varrho\, \varTheta\times\mathbf V_2\,{\rm d}x$, 
$\int_{B}\varrho\, \varTheta\cdot\mathbf W_2\,{\rm d}x=-\int_B\varrho\, \varTheta\cdot\mathbf V_2\,{\rm d}x$ and $\int_{B}\varrho\, \mathbf V_1^\ast\times\mathbf W_2\,{\rm d}x=-\int_B\varrho\,\mathbf V_1^\ast\times\mathbf V_2\,{\rm d}x$
and such that $\{\varTheta|_B\cdot\mathbf e_k,\,\mathbf V_1^\ast|_B\cdot\mathbf e_k,\,(\mathbf V_2+\mathbf W_2)|_B\cdot\mathbf e_k,\,k=1,2,3\}$ is free in $C^1_0(\mathbf R^3)^3$.
We can set $\mathbf V_2^\ast:=\mathbf V_2+\mathbf W_2$ and iterate this process to define $\mathbf V_3^\ast,\ldots,\mathbf V_n^\ast$.
\end{proof}
\section{Added mass matrix for particular shaped swimmers}
Let $\mathcal B$ be an open, bounded, connected, $C^1$ subset of $\mathbf R^3$. Denote $\varSigma$ its boundary, $\mathbf n$ the unit vector to $\varSigma$ directed toward the interior of $\mathcal B$, $\mathcal F:=\mathbf R^3\setminus\bar{\mathcal B}$ and consider the $6\times 6$ symmetric matrix $\mathbb M^f$ of which the entries are defined by $\int_{\mathcal F}\nabla\psi_i\cdot\nabla\psi_j\,{\rm d}x$ ($1\leq i,j\leq 6$)
where the functions $\psi_i$ ($i=1,\ldots,6$) are harmonic in $\mathcal F$ and satisfy the Neumann boundary conditions
$\partial_n\psi_i=(\mathbf e_i\times x)\cdot\mathbf n$ on $\varSigma$ if $i=1,\ldots,6$ and $\partial_n\psi_i=\mathbf e_{i-3}\cdot\mathbf n$ if $i=4,5,6$. 
\begin{proposition}
\label{matrix_positive_definite}
The matrix $\mathbb M^f$ is always positive and it is  positive definite if and only if $\mathcal B$ has no axis of symmetry.
\end{proposition}
\begin{proof}
Denote $\boldsymbol\alpha:=(\alpha_1,\ldots,\alpha_6)^t$ any element in $\mathbf R^6$ and $\psi:=\sum_{i=1}^6\alpha_i\psi_i$. Then, we have $\boldsymbol\alpha^t\mathbb M^f\boldsymbol\alpha=\int_{\mathcal F}\|\nabla\psi\|_{\mathbf R^3}^2{\rm d}x\geq 0$
which proves that $\mathbb M^f$ is indeed positive. Let now $\boldsymbol\alpha$ be in $\mathbf R^6$ such that $\boldsymbol\alpha^t\mathbb M^f\boldsymbol\alpha=0$. It means that $\psi=0$ and hence $\partial_n\psi=\sum_{i=1}^3\alpha_i(\mathbf e_i\times x)\cdot\mathbf n+\sum_{i=4}^6\alpha_i\mathbf e_i\cdot\mathbf n=0$. Denoting $\mathbf u:=\sum_{i=1}^3\alpha_i\mathbf e_i$ and $\mathbf v:=\sum_{i=4}^6\alpha_i\mathbf e_i\cdot\mathbf n=0$, this condition reads $(\mathbf u\times x+\mathbf v)\cdot\mathbf n=0$ on $\varSigma$. The case $\mathbf u=\mathbf 0$ is obviously not possible, otherwise we would have $\mathbf v\cdot\mathbf n(x)=0$ for all $x\in\varSigma$ so let us assume from now on that $\mathbf u\neq \mathbf 0$. For all $x\in\varSigma$, $(\mathbf u\times x+\mathbf v)\in T_x\varSigma$ (the tangent space to $\varSigma$ at $x$). Let us set $\mathbf w:=(\mathbf u\times \mathbf v)/|\mathbf u|^2$. Then $(\mathbf u\times x+\mathbf v)=\mathbf u\times(x-\mathbf w)+\lambda\mathbf u$ with $\lambda:=(\mathbf v\cdot \mathbf u)/|\mathbf u|^2$.  In this form, we can explicitly compute the expression of the flow connecting to the ODE $\dot X(t,x)=\mathbf u\times (X(t,x)-\mathbf w)+\lambda\mathbf u$, $X(0,x)=x$. We obtain $X(t,x)=\exp(t\hat{\mathbf u})(x-\mathbf w)+\mathbf w+t\lambda\mathbf u$ for all $t\in\mathbf R$. Since $X(t,x)$ has to remain on $\varSigma$ for all $t\in\mathbf R$ and $\varSigma$ is bounded, we deduce that $\lambda\mathbf u=\mathbf 0$. For all $x\in\varSigma$ and all $t\in\mathbf R$, the point $X(t,x)$ lies on $\varSigma$. It means that $\varSigma$ is globally invariant under a rotation whose axis has $\mathbf u$ as direction vector and goes through $\mathbf w$. 
\end{proof}

Let $\mathcal M$ be a smooth embedded submanifold of a Banach space $E$ and $\mathcal X:=(\mathbf X_i)_{1\leq i\leq n}$ be set of smooth vector fields on $\mathcal M$. 
\begin{lemma}
\label{LEM:converg}
Let $\alpha_i$ be in $L^\infty([0,T])$ ($i=1,\ldots,n$) and $x$ be a Carath\'eodory's solution defined on the time interval $[0,T]$ to the ODE $\dot x(t)=\sum_{i=1}^n\alpha_i(t)\mathbf X_i(x)$ ($0<t\leq T$),
with Cauchy data $x(0)=x_0\in\mathcal M$. Let the functions $\alpha^k_i$ , $(i=1,\ldots,n,\,k\in\mathbf N$) be in $L^1([0,T])$, such that $\alpha_i^k\to \alpha_i$ in $L^1([0,T])$ as $k\to+\infty$. Then, for any sequence of Carath\'eodory's solutions $(x_k)_{k\in\mathbf N}$ satisfying $\dot x_k(t)=\sum_{i=1}^n\alpha^k_i(t)\mathbf X_i(x)$
with Cauchy data $x_k(0)=x_0$, the functions $x_k$ can be continued on the whole interval $[0,T]$ for $k$ large enough, $x_k\to x$ uniformly on $[0,T]$ and $\dot x_k\to \dot x$ in $L^1([0,T],\mathcal M)$.
\end{lemma}
\begin{proof}
For any $\delta>0$ small enough, denote by $K_{\delta}$ the compact $\{x\in\mathcal M\,:\,\|x-x(t)\|_E\leq \delta,\,t\in[0,T]\}$ and denote $k_\delta>0$ the Lipschitz constant such that $\sum_{i=1}^n\|\mathbf X_i(x)-\mathbf X_i(y)\|_E<k_{\delta}\|x-y\|_E$ for all $x,\,y\in K_\delta$ and all $i=1,\ldots,n$. Let $M:=\max_{x\in K_\delta\atop i=1,\ldots,n}\|\mathbf X_i(x)\|_E$ and $m:=\sup_{t\in[0,T]\atop i=1,\ldots,n}|\alpha_i(t)|$. Any function $x_k$ is defined at least on a small interval $[0,t_k[$ and we can choose $t_k$ small enough such that $x_k(t)\in K_\delta$ for all $t\in[0,t_k[$. We get the estimate, for all $t\in[0,t_k[$:
\begin{equation}
\|x(t)-x_k(t)\|_E\leq M\sum_{i=1}^n\int_0^t|\alpha_i(s)-\alpha_i^k(s)|{\rm d}s+mk_\delta\int_0^t\|x_k(s)-x(s)\|_E{\rm d}s.\label{grn}
\end{equation}
For any $\varepsilon>0$, we can choose $k$ large enough such that $\sum_{i=1}^n\int_0^t|\alpha_i(s)-\alpha_i^k(s)|{\rm d}s<\varepsilon e^{-m k_{\delta} T}/M$. Applying Gr\"onwall inequality to \eqref{grn}, we obtain that $\|x(t)-x_k(t)\|_E<\varepsilon$ for all $t\in[0,t_k[$. We deduce first that if $\varepsilon<\delta$, the solution $x_k$ can be continued on the whole interval $[0,T]$ and then that $x_k\to x$ uniformly as $k\to+\infty$. Writing now that:
\begin{align*}
\int_0^t\|\dot x_k(t)-\dot x(t)\|_E{\rm d}t&\leq \sum_{i=1}^n\int_0^t|\alpha^k_i(t)-\alpha_i(t)|\|\mathbf X_i(x_k(s))\|_E+|\alpha_i(t)|\|\mathbf X_i(x_k(t))-\mathbf X_i(x(t))\|_E{\rm d}t\\
&\leq M \sum_{i=1}^n\int_0^t|\alpha^k_i(t)-\alpha_i(t)|{\rm d}t+m k_\delta\int_0^t\|x_k(t)-x(t)\|_E{\rm d}t,
\end{align*}
we get the convergence of the sequence $(\dot x_k)_{k\in\mathbf N}$ to $\dot x$ in $L^1([0,T],\mathcal M)$.
\end{proof}
\bibliographystyle{abbrv}
\bibliography{control_cas_general_biblio}
\end{document}